\documentclass[usletter]{amsart}
\usepackage{color}
\usepackage{colordvi}

\usepackage{graphicx}
\usepackage{amssymb}
\usepackage{amsmath}
\usepackage{amsthm}
\usepackage{ifthen}

\newtheorem{thm}{Theorem}[section]
\newtheorem{lemma}[thm]{Lemma}
\newtheorem{corollary}[thm]{Corollary}
\newtheorem{cor}[thm]{Corollary}
\newtheorem{theorem}[thm]{Theorem}
\newtheorem{prop}[thm]{Proposition}
\newtheorem{proposition}[thm]{Proposition}

\theoremstyle{definition}
\newtheorem{definition}[thm]{Definition}

\newcommand{\titre}[1]{\noindent \textbf{#1}}
\newcommand{\ite}{\noindent $\bullet~$}

\newcommand{\tth}{\textrm{th}}
\newcommand{\bu}{\bullet} 
\newcommand{\st}{\star}

\newcommand{\al}{\alpha}
\newcommand{\be}{\beta}
\newcommand{\la}{\lambda}
\newcommand{\si}{\sigma} 
 
\newcommand{\gam}{\gamma} 
 
\newcommand{\Om}{\Omega}

 \newcommand{\tbeta}{\tilde \beta} 
\newcommand{\tbe}{\tilde \beta} 
\newcommand{\ttau}{\tilde \tau} 
\newcommand{\tpi}{\tilde \pi} 
\newcommand{\tsi}{\tilde \sigma} 
\newcommand{\tsip}{\raisebox{0pt}{$\tilde{\si}'$}} 
\newcommand{\ttaup}{\raisebox{0pt}{$\tilde{\tau}'$}} 
 
\newcommand{\tpsi}{\tilde \psi} 
\newcommand{\tphi}{\tilde \phi}
\newcommand{\tphip}{\raisebox{0pt}{$\tilde{\phi}'$}} 
\newcommand{\tvarphi}{\tilde \varphi}
\newcommand{\tvarphip}{\raisebox{0pt}{$\tilde{\varphi}'$}} 
 
\newcommand{\tA}{\tilde A} 
\newcommand{\tB}{\tilde B}

\newcommand{\tI}{\tilde I} 
\newcommand{\tO}{\tilde O} 
\newcommand{\tM}{\tilde M}

\newcommand{\ba}{\bar{a}}
\newcommand{\bb}{\bar{b}}
\newcommand{\lS}{\prec_S}
\newcommand{\lSp}{\prec_{S'}}
\newcommand{\bS}{\bar{S}}
\newcommand{\MS}{M_{|S}}
\newcommand{\MSs}{M^*_{|\bS}}

\newcommand{\Cat}{\mathrm{Cat}}

\usepackage{comment}

\author[O. Bernardi]{Olivier Bernardi $^1$}
\thanks{$^1$ Massachussets Institute of Technology,  Department of Mathematics,  Cambridge MA 02139, USA. Partial support from the French grant \emph{ANR A3} and the European grant \emph{ERC Explore maps}.  {\tt bernardi@math.mit.edu}}

\author[G. Chapuy]{Guillaume Chapuy $^2$}
\thanks{$^2$Department of Mathematics, Simon Fraser University, Burnaby BC V5A 1S6, Canada. 
Partial support from the a CNRS/PIMS postdoctoral fellowship, and the European grant \emph{ERC Explore maps}. 
{\tt gchapuy@sfu.ca}
}

\title[A bijection for covered maps]
{A bijection for covered maps, or a shortcut between Harer-Zagier's and Jackson's formulas.}

\date{\today}

\begin{document}

\maketitle

\begin{abstract} We consider maps on orientable surfaces. A map is called \emph{unicellular} if it has a single face. A \emph{covered map} is a map (of genus $g$) with a marked unicellular spanning submap (which can  have any genus in $\{0,1,\ldots,g\}$).  Our main result is a bijection between covered maps with $n$ edges and genus $g$ and pairs made of a plane tree with $n$ edges and a unicellular bipartite map of genus $g$ with $n+1$ edges. In the planar case, covered maps are maps with a marked spanning tree  and our bijection specializes into a construction obtained by the first author in~\cite{OB:boisees}.

Covered maps can also be seen as \emph{shuffles} of two unicellular maps (one representing the unicellular submap, the other representing the dual unicellular submap). Thus, our bijection gives a correspondence between shuffles of unicellular maps, and pairs made of a plane tree and a unicellular bipartite map. In terms of counting, this establishes the equivalence between a formula due to Harer and Zagier for general unicellular maps, and a formula due to Jackson for bipartite unicellular maps.

We also show that the bijection of Bouttier, Di Francesco and Guitter~\cite{BDFG:mobiles} (which  generalizes a previous bijection by Schaeffer~\cite{Schaeffer:these}) between bipartite maps and so-called well-labelled mobiles can be obtained as a special case of our bijection.
\end{abstract}

\section{Introduction.} \label{sec:intro}
We consider maps on orientable surfaces of arbitrary genus. A map is called \emph{unicellular} if it has a single face. For instance, the unicellular maps of genus 0 are the plane trees. A \emph{covered map} is a map together with a marked unicellular spanning submap. A map of genus $g$ can have spanning submaps of any genus in $\{0\ldots,g\}$. In particular, a \emph{tree-rooted map} (map with a marked spanning tree) is a covered map since the spanning trees are the unicellular spanning submaps of genus~0. A covered map of genus 2 with a unicellular spanning submap of genus~1 is represented in Figure~\ref{fig:Covered-map+image}(a).\\ 

Our main result is a bijection (denoted by $\Psi$) between covered maps of genus $g$ with $n$ edges, and pairs made of a plane tree with $n$ edges and a bipartite unicellular map of genus $g$ with $n+1$ edges. If the covered map has $v$ vertices and $f$ faces, then the bipartite unicellular map has $v$ white vertices and $f$ black vertices. The bijection $\Psi$ is represented in Figure~\ref{fig:Covered-map+image}. In the planar case $g=0$, the bijection $\Psi$ coincides with a construction by the first author~\cite{OB:boisees} between planar tree-rooted maps with $n$ edges and pairs of plane trees with $n$ and $n+1$ edges respectively\footnote{In \cite{OB:boisees}, the tree with $n+1$ edges was actually described as a non-crossing partition.}. As explained below, the bijection $\Psi$ also extends several other bijections. 

\begin{figure}[ht!]\begin{center} \input{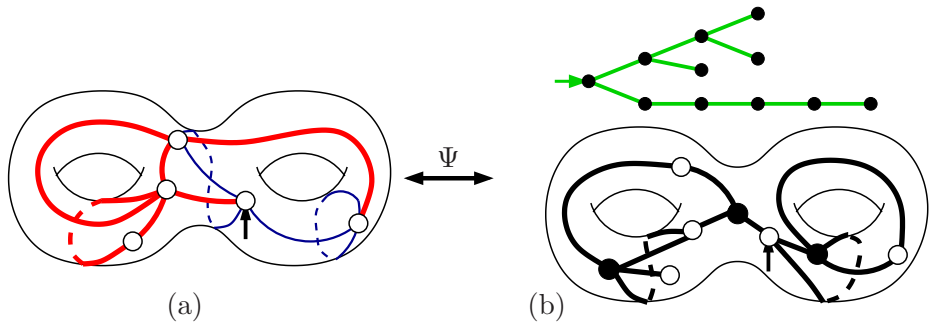}\caption{(a) A covered map of genus 2 (the unicellular submap of genus 1 is drawn in thick lines). (b) The image of the covered map by the bijection $\Psi$ is made of a tree and a bipartite unicellular map of genus 2 (the \emph{mobile} associated to the covered map).}\label{fig:Covered-map+image} \end{center}\end{figure}

Before discussing the bijection $\Psi$ further, let us give another equivalent way of defining covered maps. We start with the planar case which is simpler. Let $M$ be a planar map and let $M^*$ be its dual. Then, for any spanning tree $T$ of $M$, the dual of the edges not in $T$ form a spanning tree of $M^*$. In other words, the \emph{dual of a planar tree-rooted map is a planar tree-rooted map}.  Pushing this observation further, Mullin showed in~\cite{Mullin:tree-rooted-maps} that a tree-rooted map can be encoded by a \emph{shuffle} of two trees (one representing the spanning tree on primal map $M$, the other representing the spanning tree on the dual map $M^*$), or more precisely as a shuffle of two parenthesis systems encoding these trees. Covered maps enjoy a similar property:  the \emph{dual of a covered map is a covered map}. Using this observation, it is not hard to see that covered maps can be encoded by \emph{shuffles} of two unicellular maps (see Section~\ref{sec:shuffles}). 

We emphasize that our bijection $\Psi$ is \emph{not} the encoding of covered maps as shuffles of two unicellular maps: the image by $\Psi$ is a pair of unicellular maps of a fixed size, and not a shuffle.
As a matter of fact, comparing the enumerative formulas given by the bijection $\Psi$ with the formulas given by the shuffle approach yields the equivalence between formulas by Harer and Zagier~\cite{Harer-Zagier}, and by Jackson~\cite{Jackson:kpartite-unicellular-maps}. As a warm up, let us consider the total number $C(n)$ of covered maps with $n$ edges (and arbitrary genus). 
The bijection $\Psi$ implies  
\begin{eqnarray}\label{eq:pair-arbitrary-genus}
C(n)=\Cat(n)B(n+1)=\frac{(2n)!}{n!},
\end{eqnarray}
where $\Cat(n)=\frac{(2n)!}{n!(n+1)!}$ is the Catalan number counting rooted plane trees with $n$ edges, and $B(n+1)=(n+1)!$ is the total number of bipartite unicellular maps with $n+1$ edges. On the other hand, the shuffle approach implies
\begin{eqnarray}\label{eq:shuffle-arbitrary-genus}
C(n)=\sum_{k=0}^n{2n\choose 2k} A(k) A(n-k),
\end{eqnarray}
where $A(k)=\frac{(2k)!}{2^kk!}$ is the total number of unicellular maps with $k$ edges (and the term ${2n\choose 2k}$ accounts for the shuffling). Here, the relation between \eqref{eq:pair-arbitrary-genus} and \eqref{eq:shuffle-arbitrary-genus} is merely an application of the Chu-Vandermonde identity, but things get more interesting when one considers refinements of these equations. First, by the bijection~$\Psi$, the number of covered maps of genus $g$ with $n$ edges is
\begin{eqnarray}\label{eq:boiseesintro}
C_g(n) = \Cat(n) B_g(n+1),
\end{eqnarray}
where $B_g(n+1)$ is the number of bipartite unicellular maps of genus $g$ with $n+1$ edges. A further refinement can be obtained by taking into account the numbers $p$ and $q$ of vertices and faces of the covered map. We show in Section \ref{sec:enumeration}, that comparing the formula given by the bijection $\Psi$ with the formula given  by the shuffle approach yields
\begin{eqnarray}
\label{eq:bipartitefrommono-intro}
B^{p,q}(n+1) = \sum_{k=0}^n \frac{n!(n+1)!}{(2k)!(2n-2k)!} A^p(k)A^q(n-k),
\end{eqnarray}
where $B^{p,q}(n+1)$ is the number of unicellular bipartite maps with $n+1$ edges, $p$ white vertices and $q$ black vertices, and $A^p(k)$ is the number of unicellular maps with $n$ edges and $p$ vertices. Equation \eqref{eq:bipartitefrommono-intro} actually establishes the  equivalence between the formula of Harer and Zagier~\cite{Harer-Zagier} for unicellular maps: 
\begin{eqnarray}
\label{eq:HarerZagier}
\sum_{p\geq 1} A^p(n) y^p = \frac{(2n)!}{2^nn!} \sum_{i\geq 1}2^{i-1} {n \choose i-1} {y \choose i},
\end{eqnarray}
and the formula of Jackson~\cite{Jackson:kpartite-unicellular-maps} for bipartite unicellular maps:
\begin{eqnarray}
\label{eq:JacksonAdrianov}
\sum_{p,q\geq 1} B^{p,q}(n+1) y^p z^q= (n+1)! \sum_{i,j\geq 1} {n \choose i-1, j-1} {y \choose i} {z \choose j}.
\end{eqnarray}
The original proof of \eqref{eq:HarerZagier} involves a matrix integral argument; combinatorial proofs are given in \cite{Lass:Harer-Zagier,Goulden:Harer-Zagier,OB:Harer-Zagier-non-orientable,Chapuy:unicellular}. The original proof of \eqref{eq:JacksonAdrianov} (as well as another related formula by Adrianov~\cite{Adrianov:bicolored-unicellular}) is based on the representation theory of the symmetric group; a combinatorial proof is given in  \cite{Schaeffer:bij-Jackson-formula}.

Let us mention a few other enumerative corollaries of the bijection $\Psi$.  
First, plugging the identity\footnote{This identity is originally due to Lehman and Walsh~\cite[Eq. (14)]{Walsh:counting-maps-1} and can easily be proved bijectively by using the results from~\cite{Chapuy:unicellular}.} $A^1(k)=A(k)/(k+1)$  in the case $q=1$ of
Equation~\eqref{eq:bipartitefrommono-intro} shows that the number $B^{p}(n)$ of bipartite  unicellular
maps with $n$ edges and $p$ white vertices satisfies
$$B^{p}(n)=\frac{n(n+1)}{2}B^{p,1}(n).$$
This formula is originally due to Zagier~\cite{Zagier:cycles}, and has been given a bijective proof (different
from ours) by Feray and Vassilieva~\cite{Feray-Vassilieva}.
We now turn to formulas concerning the number $T_g(n)$ of tree-rooted maps of genus $g$ with $n$ edges. In the planar case, tree-rooted maps are the same as covered maps. Hence, \eqref{eq:boiseesintro} gives
\begin{eqnarray}\label{eq:Mullin}
T_0(n)= \Cat(n) \Cat(n+1),
\end{eqnarray}
This formula was originally proved by Mullin~\cite{Mullin:tree-rooted-maps} (using the shuffle approach) who asked for a direct bijective proof; this was the original motivation for the planar specialization of $\Psi$ described in~\cite{OB:boisees} 
(and for a related recursive bijection due to Cori, Dulucq and Viennot~\cite{Dulucq:shuffle-parenthesis-system}). 
In the case of genus $g=1$, a duality argument shows that exactly half of the covered maps are tree-rooted maps, so that \eqref{eq:boiseesintro} gives
\begin{eqnarray}\label{eq:Lehman}
T_1(n)= \frac{1}{2}\Cat(n) B_1(n+1)=\frac{(2n)!(2n-1)!}{12(n+1)!n!(n-1)!(n-2)!},
\end{eqnarray}
This formula was originally proved by Lehman and Walsh (using the shuffle approach), and had no direct bijective proof so far.\\


We now explain the relation between the bijection $\Psi$ and some known bijections. In \cite{Bernardi-Fusy:dangulations,Bernardi-Fusy:Bijection-girth}, two ``master bijections'' for planar maps are defined, and then specialized in various ways so as to unify and extend several known bijections (roughly speaking, by specializing the master bijections appropriately, one can obtain a bijection for any class of planar maps defined by a girth condition and a face-degree condition). The master bijections are in fact ``variants'' of the planar version of the bijection $\Psi$. 
Here we shall prove that the bijection $\Psi$ generalizes the bijection obtained for bipartite planar maps by Bouttier, Di Francesco and Guitter~\cite[Section 2]{BDFG:mobiles}  (however, we do not recover the most general version of their bijection~\cite[Section 3-4]{BDFG:mobiles}), as well as its generalization to higher genus surfaces by Chapuy, Marcus and Schaeffer~\cite{Chapuy:nb-maps-orientable, Chapuy:constellations}. These bijections (which themselves generalize a previous bijection by Schaeffer~\cite{Schaeffer:these}) are of fundamental importance for studying the metric properties of random maps~\cite{Chassaing-Schaeffer:ISE,BDFG:distances,Bouttier:3point-function,Bouttier:Godesic,Miermont:Geodesiques} and for defining and analyzing their continuous limit, the \emph{Brownian map}~\cite{Marckert:limit-quadrangulations,LeGall:limitmaps,LeGallPaulin:limitmaps2,LeGall:Geodesiques}.\\

The paper is organized as follows. In Section~\ref{sec:Maps}, we recall some definitions about maps. In Section~\ref{sec:shuffles}, we show that covered maps can be encoded by shuffles of unicellular maps. In Section~\ref{sec:bijection}, we define the bijection $\Psi$ between covered maps with $n$ edges and pairs made of a plane tree with $n$ edges and a bipartite unicellular map with $n+1$ edges. In Section~\ref{sec:enumeration}, we explore the enumerative implications of the bijection $\Psi$. In Section~\ref{sec:proofs}, we prove the bijectivity of $\Psi$. In Section~\ref{sec:alternative-folding}, we give three equivalent ways of describing the image of $\Psi$ (the pairs made of a plane tree and a bipartite unicellular map) and use one  of these descriptions in order to show that the bijections for bipartite maps described in~\cite{BDFG:mobiles,Chapuy:nb-maps-orientable} are specializations of $\Psi$.  Lastly, in Section~\ref{sec:duality}, we study the properties of the bijection $\Psi$ with respect to duality.\\


\section{Definitions} \label{sec:Maps}
\titre{Maps.} Maps can either be defined topologically (as graphs embedded in surfaces) or combinatorially (in terms of permutations). We shall prove our results using the combinatorial definition, but resort to the topological interpretation in order to convey intuitions.\\

We start with the topological definition of maps. Here, \emph{surfaces} are $2$-dimensional, oriented, compact and without boundaries. A \emph{map} is a connected graph embedded in surface, considered up to orientation preserving homeomorphism. By \emph{embedded}, one means drawn on the surface in such a way that the edges do not intersect and the \emph{faces} (connected components of the complement of the graph) are simply connected. Loops and multiple edges are allowed. The \emph{genus} of the map is the genus of the underlying surface and its \emph{size} is its number of edges. A \emph{planar map} is a map of genus $0$. A map is \emph{unicellular} if it has a single face. For instance, the planar unicellular maps are the \emph{plane trees}. A map is \emph{bipartite} if vertices can be colored in black and white in such a way that every edge join a white vertex to a black vertex. We denote by $g(M)$ the genus of a map $M$ and by $v(M)$, $f(M)$, $e(M)$ respectively its number of vertices, faces and edges. These quantities are related by the \emph{Euler formula}:
\begin{eqnarray*}
\label{eq:Euler}
v(M)-e(M)+f(M)=2-2g(M).
\end{eqnarray*}
By removing the midpoint of an edge, one obtains two \emph{half-edges}. Two consecutive half-edges around a vertex define a \emph{corner}. A map is \emph{rooted} if one half-edge is distinguished as the \emph{root}. The vertex incident to the root is called \emph{root-vertex}. In figures, the rooting will be indicated by an arrow pointing into the \emph{root-corner}, that is, the corner following the root in clockwise order around the root-vertex. For instance, the root of the map in Figure~\ref{fig:Def-map} is the half-edge $a_1$. \\

\begin{figure}[ht!]\begin{center} \input{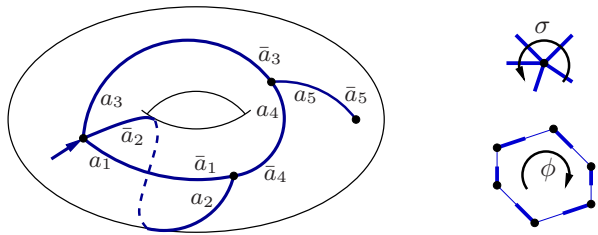}\caption{A unicellular map $M=(H,\si,\al)$ of genus 1. The set of half-edges is $H=\{a_1,\ba_1,\ldots,a_5,\ba_5\}$, the edge-permutation is $\al=(a_1,\ba_1)\cdots(a_5,\ba_5)$, the vertex-permutation  is  $\si=(a_1,\ba_2,a_3)(\ba_1,a_2,\ba_4)(\ba_3,a_4,a_5)(\ba_5)$, and the face-permutation is $\phi=\si\alpha=(a_1,a_2,a_3,a_4,\ba_1,\ba_2,\ba_4,a_5,\ba_5,\ba_3)$.}\label{fig:Def-map} \end{center}\end{figure}

Maps can also be defined in terms of permutations acting on half-edges. 
To obtain this equivalence, observe first that the embedding of a graph in a surface defines a cyclic order (the counterclockwise order) of the half-edges around each vertex. This gives in fact a one-to-one correspondence between maps and connected graphs together
with a cyclic order of the half-edges around each vertex
(see e.g.~\cite{Mohar-Thomassen}). Equivalently, a \emph{map} can be defined as a triple
$M=(H,\si,\alpha)$ where $H$ is a finite set whose elements are called the
\emph{half-edges}, $\alpha$ is an involution of $H$ without fixed point, and
$\si$ is a permutation of $H$ such that the group generated by $\si$ and
$\alpha$ acts transitively on $H$ (here we follow the notations in~\cite{Cori-Machi}). This must be understood as follows: each cycle of $\si$ describes the counterclockwise order of the half-edges around one vertex of the map, and each cycle of $\alpha$ describes an edge, that is,  a pair of two half-edges; see Figure~\ref{fig:Def-map} for an example. The transitivity assumption simply translates the fact that the graph is connected.  \\

For a map $M=(H,\si,\alpha)$, the permutation $\si$ is called \emph{vertex-permutation}, the permutation $\alpha$  is called  \emph{edge-permutation} and the permutation $\phi=\si\alpha$ is called \emph{face-permutation}. The cycles of $\si$, $\al$, $\phi$ are called \emph{vertices}, \emph{edges} and \emph{faces}. Observe that the cycles of $\phi$ are indeed in bijection with the faces of the map in its topological interpretation. Hence, the genus of $M$ can be deduced from the number of cycles of $\si$, $\al$ and $\phi$ by the Euler relation. We say that a half-edge is \emph{incident} to a vertex or a face if this edge belongs to the corresponding cycle.  Again, a map is \emph{rooted} if one of the half-edges is distinguished as the \emph{root}; the incident vertex and face are called \emph{root-vertex} and \emph{root-face}.\\

The correspondence between topological and combinatorial maps is one-to-one if combinatorial maps are considered up to \emph{isomorphism} (or,  \emph{relabelling}). That is, two maps $(H,\si,\alpha)$ and $(H',\si',\alpha')$ are considered the same if there exists a bijection
$\la: H \rightarrow H'$ such that $\si'=\la\si\la^{-1}$ and $\alpha'=\la\alpha\la^{-1}$ (for rooted maps, we ask furthermore that $\la(r)=r'$).  In this article \emph{all maps will be rooted, and considered up to isomorphism}.

We call \emph{pseudo map} a triple $M=(H,\si,\alpha)$ such that $\alpha$ is a fixed-point free involution, but where the  transitivity assumption (i.e. connectivity assumption) is not required. This can be seen as a union of maps and we still call $\phi=\si\al$ the \emph{face-permutation}, as its cycles are indeed in correspondence with the faces of the union of maps.  Lastly, we consider the case where the set of half-edges $H$ is empty as a special case of rooted unicellular map (corresponding to the planar map with one vertex and no edge) called \emph{empty map}.\\

\titre{Submaps, covered maps and motion functions.}
For a permutation $\pi$ on a set $H$, we call \emph{restriction of $\pi$} to a set $S\subseteq H$ and denote by $\pi_{|S}$ the permutation of $S$ whose cycles are obtained from the cycles of $\pi$ by erasing the elements not in $S$. Observe that $(\pi^{-1})_{|S}=(\pi_{|S})^{-1}$ so that we shall not use parenthesis anymore in these notations. It is sometime convenient to consider the restriction $\pi_{|S}$ as a permutation on the whole set $H$ acting as the identity on $H\setminus S$; we shall mention this abuse of notations whenever necessary.\\ 

A spanned map is a map with a marked subset of edges. 
In terms of permutations, a \emph{spanned map} is a pair $(M,S)$, where $S$ is a subset of half-edges stable by the edge-permutation $\al$. The  \emph{submap defined by $S$}, denoted $M_{|S}$, is the pseudo map $(S,\si_{|S},\alpha_{|S})$, where $\si$ is the vertex-permutation of $M$. We underline that the face-permutation $\phi_S=\si_{|S}\alpha_{|S}$ of the pseudo-map $M_{|S}$ is not equal to  $(\si\alpha)_{|S}$. Observe also that the genus of $\MS$ can be less than the genus of $M$. For example, Figure~\ref{fig:Covered-map+image}(a) represents a submap of genus 1 of a map of genus 2.
A submap $M_{|S}$ is \emph{connecting} if it is a map containing every vertex of $M$, that is, $S$ contains a half-edge in every vertex of $M$ (except if $M$ has a single vertex, where we authorize $S$ to be empty) and $\si_{|S}$, $\alpha_{|S}$ act transitively on $S$. The submap represented in Figure~\ref{fig:motion} (right) is a map but is not connecting. A \emph{covered map} is a spanned map such that the submap $M_{|S}$ is a connecting unicellular map. A \emph{tree-rooted map} is a spanned map such that the submap $M_{|S}$ is a spanning tree, that is, a connecting plane tree.\\ 

The \emph{motion function} of the spanned map $(M,S)$ is the mapping $\theta$ defined on $H$ by $\theta(h)=\phi(h)\equiv\si\alpha(h)$ if $h$ is in $S$ and $\theta(h)=\si(h)$ otherwise. Observe that the  stability of $S$ by $\alpha$ implies that the motion function $\theta$ is a permutation of $H$ since its inverse is given by $\theta^{-1}(h)=\alpha\si^{-1}(h)$ if $\si^{-1}(h)$ is in $S$ and $\theta^{-1}(h)=\si^{-1}(h)$ otherwise. Observe also that, given the map $M$,  the set $S$ can be recovered from the motion function $\theta$. Topologically, the motion function is the permutation describing the \emph{tour} of the faces of the connected components of the submap $\MS$ in counterclockwise direction: we follow the border of the edges of the submap $\MS$ and cross the edges not in $\MS$. See Figure~\ref{fig:motion} for an example.\\

\begin{figure}[ht!]\begin{center} \input{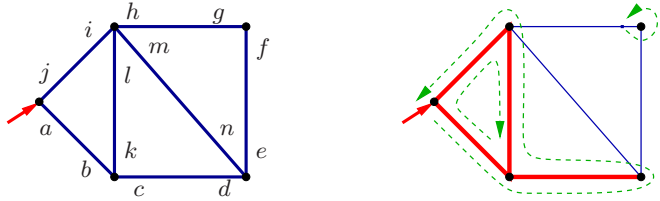}\caption{The submap $\MS$ defined by $S=\{a,b,c,d,i,j,k,l\}$, and its motion function $\theta=(a,c,e,n,d,k,m,h,i)(b,j,l)(f,g)$. The cycles of $\theta$ are represented in dashed lines.}\label{fig:motion} \end{center}\end{figure}

\titre{Orientations.}
An \emph{orientation} of a map $M=(H,\si,\alpha)$ is a partition $H=I\uplus O$ such that the involution $\alpha$ maps the set $I$ of \emph{ingoing} half-edges to the set $O$ of \emph{outgoing} half-edges. The pair $(M,(I,O))$ is an \emph{oriented map}. 
A \emph{directed path} is a sequence $h_1,h_2,\ldots,h_{k}$ of distinct ingoing half-edges such that $h_{i}$, $\al(h_{i+1})$ are incident to the same vertex (are in the same cycles of $\si$) for  $i=1\ldots k\!-\!1$. A directed cycle is a directed path $h_1,\ldots,h_{k}$ such that $h_{k}$ and $\al(h_1)$ are incident to the same vertex. The half-edge $h_k$ is called the \emph{extremity} of the directed path. An orientation is \emph{root-connected} if any ingoing half-edge $h$ is the extremity of a directed path $h_1,\ldots,h_{k}=h$ such that $\al(h_1)$ is incident to the root-vertex of~$M$.\\

\titre{Duality}
The \emph{dual map} of a map $M=(H,\si,\alpha)$ is the map $M^*=(H,\phi,\alpha)$ where $\phi=\si\alpha$ is the face-permutation of $M$. The root of the dual map $M^*$ is equal to the root of~$M$. Observe that the genus of a map and of its dual are equal (by Euler relation) and that $M^{**}=M$. Topologically, the dual map $M^*$ is obtained by the following two steps process: see Figure~\ref{fig:dual+covered}.
\begin{enumerate}
\item In each face $f$ of $M$, draw a vertex $v_f$ of $M^*$. For each edge $e$ of $M$ separating faces $f$ and $f'$ (which can be equal), draw the \emph{dual edge} $e^*$ of $M^*$ going from $v_f$ to $v_{f'}$ across $e$.
\item Flip the drawing of $M^*$, that is, inverse the orientation of the surface.\\
\end{enumerate}

We now define duals of spanned maps and oriented maps. Given a subset $S\subseteq H$, we denote $\bS=H\setminus S$. The \emph{dual of a spanned map} $(M,S)$ is the spanned map $(M^*,\bS)$; see Figure~\ref{fig:dual+covered}. We also say that $M_{|S}$ and $M^*_{|\bS}$ are \emph{dual submaps}. Observe that the motion functions of a spanned map $(M,S)$ and of its dual $(M^*,\bS)$ are equal.
The \emph{dual of the oriented map}  $(M,(I,O))$ is $(M^*,(I,O))$. Graphically, this orientation is obtained by applying the following rule at step 1: the dual-edge $e^*$ of an edge  $e\in M$ is oriented from the left of $e$ to the right of $e$;  see Figure~\ref{fig:dual+orientation}. Observe that duality 
is involutive on maps, spanned maps and oriented maps.

\begin{figure}[ht!]\begin{center} \input{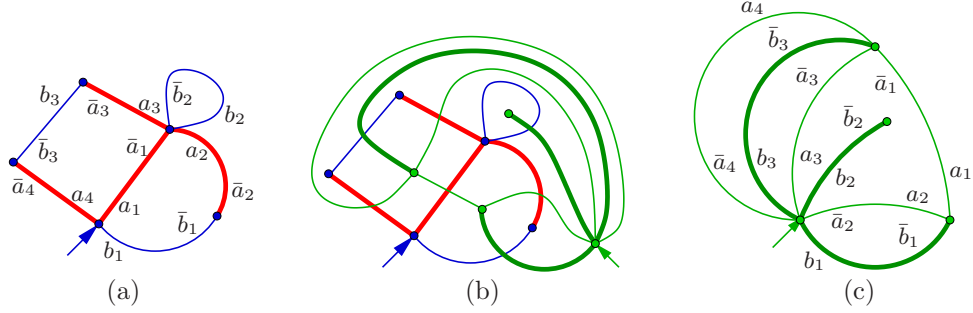}\caption{(a) A spanned map (the submap is indicated by thick lines).  (b) Topological construction of the dual. (c) The dual covered map.}\label{fig:dual+covered} \end{center}\end{figure}


\section{Covered maps as shuffles of unicellular maps.}\label{sec:shuffles}
In this Section, we establish some preliminary results about covered maps. In particular we prove that covered maps are stable by duality and explicit their encoding as shuffles of two unicellular maps. Our first result should come as no surprise: it simply states that a spanned map $(M,S)$ is a covered map if and only if turning around the submap $\MS$ (that is following the border of its edges) starting from the root allows one to visit every half-edge of $M$.

\begin{prop}\label{prop:covered=cyclic}
Let $(M,S)$ be a spanned map, and let $\si$, $\al$ and $\phi=\si\al$ be the vertex-, edge-, and face-permutations of $M$. The motion function $\theta$ satisfies $\theta_{|S}=\si_{|S}\alpha_{|S}$ and $\theta_{|\bS}=\phi_{|\bS}\alpha_{|\bS}$. That is, the restriction $\theta_{|S}$ is the face permutation of the pseudo map $\MS$, while the restriction  $\theta_{|\bS}$ is the face-permutation of the dual pseudo map $M^*_{|\bS}$. 

In particular a spanned map is a covered map if and only if its motion function is a cyclic permutation. 
\end{prop}
Proposition~\ref{prop:covered=cyclic} is a consequence of the following lemma:
\begin{lemma}\label{lem:cori}
Let $\sigma$ and $\phi$ be two permutations on the set $\{1,2,\dots,n\}$, let $S$ be a subset of $\{1,2,\dots,n\}$, and let $\theta$ be the mapping defined by:
$$
\theta(a)=\sigma(a) \mbox{ if } a\in S \mbox{ and } \theta(a)=\phi(a) \mbox{ if } a\not\in S.
$$
Then $\theta$ is a permutation if and only if $S$ is stable by $\phi^{-1}\sigma$. Moreover in that case we have:
$$
\theta_{|S}=\phi_{|S}(\phi^{-1}\sigma)_{|S}.
$$
\end{lemma}
\begin{proof}[Proof of Lemma~\ref{lem:cori}]
First, if $S$ is stable by $\phi^{-1}\sigma$ then the inverse of $\theta$ is given by $\theta^{-1}(a)=\sigma^{-1}(a)$ if $\sigma^{-1}(S)\in S$ and $\theta^{-1}(a)=\phi^{-1}(a)$ is $\sigma^{-1}(S)\not \in S$ (note that $\sigma^{-1}(S)\not \in S$ is equivalent to $\phi^{-1}(a)\not\in S$ since $S$ is stable by $\phi^{-1}\sigma$).
Conversely, if there exists $s\in S$ such that $\phi^{-1}\sigma(s)\not \in S$, then $\theta(s)=\sigma(s)=\theta(\phi^{-1}\sigma(s))$ and $\theta$ is not a permutation. This proves the first claim.

For the second claim, let $s \in S$ and $r=\theta_{|S}(s)$. By definition of the restriction we have $r\in S$, and there exist $h_1, h_2, \dots, h_k \not\in S$ such that $\theta(s)=h_1$, $\theta(h_i)=h_{i+1}$ for $i<k$, and $\theta(h_k)=r$. 
Moreover by definition of $\theta$ we have $\theta(s)=\sigma(s)$ and $\theta(h_i)=\phi(h_i)$ for $i\leq k$.
Now, since $S$ is stable by $\phi^{-1}\sigma$, we have $(\phi^{-1}\sigma)_{|S}(s)=\phi^{-1}\sigma(s)=\phi^{-1}(h_1)$, which implies that 
$\phi_{|S}(\phi^{-1}\sigma)_{|S}(s)=\phi_{|S}(h_1)=r$ by definition of the restriction.
\end{proof}



\begin{proof}[Proof of Proposition~\ref{prop:covered=cyclic}]  
The fact $\theta_{|S}=\sigma_{|S}\alpha_{|S}$ comes from Lemma~\ref{lem:cori}, and the relation $\bar\theta_{|\bar S}=\phi_{|\bar S}\alpha_{|\bar S}$ follows from the preceding point by duality (since the motion functions of a spanned map and its dual are equal).

Now let $(M,S)$ be a covered map. Since $\MS$ is connecting, each cycle of the motion function $\theta$ contains an element of $S$. Hence, the number of cycles of $\theta$ and $\theta_{|S}$ is the same. Moreover, by Lemma~\ref{lem:cori}, $\theta_{|S}=\si_{|S}\al_{|S}$ is the face-permutation of $\MS$. Since $\MS$ is unicellular,  $\theta_{|S}=\si_{|S}\al_{|S}$ is cyclic and $\theta$ is also cyclic. 

Conversely, suppose that the motion function $\theta$ is cyclic. In this case, the pseudo map $\MS$ has a face-permutation which is cyclic by Lemma~\ref{lem:cori}. Hence it is a unicellular map. 
\end{proof}

Proposition~\ref{prop:covered=cyclic} immediately gives the following corollary concerning duality.
\begin{cor}\label{cor:dual-covered}
If a spanned map $(M,S)$ is a covered map, then the dual spanned map $(M^*,\bS)$ is also a covered map. Moreover the genus of  $M$ is the sum of the genera of the unicellular maps $M_{|S}$ and $M^*_{|\bS}$:
\begin{equation}
g(M)=g(M_{|S})+g(M^*_{|\bS}).\nonumber
\end{equation}
\end{cor}
 
Corollary~\ref{cor:dual-covered} is illustrated by Figure~\ref{fig:dual+covered}. 

\begin{proof} 
The fact that $(M^*,\bS)$ is a covered map is an immediate consequence of Proposition~\ref{prop:covered=cyclic} since the motion function of a submap and of its dual are always equal. The fact that the genus adds up is obtained by writing the Euler relation for the maps $M$, $M_{|S}$ and $M^*_{|\bS}$.
\end{proof}

Let $(M,S)$ be a covered map.  By Proposition~\ref{prop:covered=cyclic}, the restrictions $\theta_{|S}$ and $\theta_{|\bS}$ of the motion function $\theta$ correspond respectively to the face-permutations of the unicellular maps $\MS$ and $M^*_{|\bS}$. This inclines to say, somewhat vaguely, that \emph{the covered map $(M,S)$ is a shuffle of  the unicellular maps $\MS$ and $\MSs$}. Making this statement precise requires introducing \emph{codes} of unicellular maps and covered maps.\\

A \emph{unicellular code} on the alphabet  $A_n=\{a_1,\bar{a}_1,\ldots,a_n,\bar{a}_n\}$ is a word on $A_n$ such that every letter of $A_n$ appears exactly once, and for all $1\leq i<j\leq n$, the letter $a_i$ appears before $\bar{a}_i$ and before $a_j$. Let $T=(H,\si,\al)$ be a unicellular map with $n$ edges. By definition, the face-permutation $\phi=\si\al$ is cyclic. Hence, there exists a unique way of relabelling the half-edges on the set $A_n$ in such a way that $\al(a_i)=\bar{a}_i$ for all $i=1\ldots n$ and $\phi=(w_1,w_2,\ldots,w_{2n})$, where $w_1$ is the root and $w=w_1w_2\cdots w_{2n}$ is a unicellular code. We call $w$ the \emph{code} of the unicellular map $T$. \\

Topologically, the code of a unicellular map is obtained by turning around the face of the map in counterclockwise direction starting from the root and writing $a_i$ when we discover the $i$th edge and writing $\bar{a}_i$ when we see this edge for the second time. For instance, the code of the unicellular map in Figure~\ref{fig:Def-map} is $w=a_1a_2a_3a_4\ba_1\ba_2\ba_4a_5\ba_5\ba_3$. We also mention that the unicellular map $T$ is a plane tree if and only if its code $w$ does not contain a subword of the form $a_ia_j\bar{a}_i\bar{a}_j$. In this special case, replacing all the letters $a_i,i=1\ldots n$ of the code $w$ by the letter $a$ and all the letters $\bar{a}_i,i=1\ldots n$ by the letter $\bar{a}$ results in no loss of information. One thereby obtains the classical bijection between plane trees and parenthesis systems on $\{a,\bar{a}\}$.

\begin{lemma}[Folklore]
The mapping which associates its code to a unicellular map is a bijection between unicellular map with $n$ edges and unicellular code on the alphabet $A_n$.  
\end{lemma}

\begin{proof}
The mapping is injective since the root and the edge-permutation $\alpha$ and vertex-permutation $\si=\phi\al$ can be recovered from the code. 
It is also surjective since starting from any code one obtains a pair of permutation $\al,\si$ which indeed gives a unicellular map $T=(A_n,\al,\si)$ (the only non-obvious property is the transitivity condition, but this is granted by the fact the face-permutation $\phi=\si\al$ is cyclic).
\end{proof}

A word on $A_k\uplus B_l$ (where $B_l=\{b_1,\bar{b}_1,\ldots,b_l,\bar{b}_l\}$) is a \emph{code-shuffle} if  the subwords $w_{|A}$ and $w_{|B}$ made of the letters in $A_k$ and $B_l$ respectively  are unicellular codes on $A_k$ and $B_l$. Let $(M,S)$ be a covered map, where $M=(H,\si,\al)$ and let $k=|S|/2$, $\,l=|\bS|/2$. By Lemma~\ref{prop:covered=cyclic}, the motion function $\theta$ is cyclic. Hence, there exists a unique way of relabelling the half-edges on the set $A_k\uplus B_l$ in such a way that $S=A_k$, $\bS=B_l$,  $\al(a_i)=\bar{a}_i$ for all $i=1\ldots k$, $\al(b_i)=\bar{b}_i$ for all $i=1\ldots l$, and $\theta=(w_1,w_2,\ldots,w_{2n})$, where $w_1$ is the root of $M$ and $w=w_1w_2\cdots w_{2n}$ is a code-shuffle. We call $w$ the \emph{code} of the covered map $(M,S)$.\\

Topologically, the code of a covered map $(M,S)$ is obtained by turning around the submap $T=M_{|S}$ in counterclockwise direction starting from the root and writing $a_i$ (resp. $b_i$) when we discover the $i$th edge in $S$ (resp. $\bS$) and writing $\bar{a}_i$ (resp. $\bar{b}_i$) when we see this edge for the second time. For instance, the code of the covered map in Figure~\ref{fig:bijection1}(a) is $w=a_1b_1a_2b_2\ba_2b_3\ba_1\bb_1a_3b_4a_4a_5\bb_3\ba_5\bb_2\ba_4\bb_4\ba_3$. We now state the main result of this preliminary section. 

\begin{prop}\label{prop:covered=shuffle}
The mapping $\phi$ which associates its code to a covered map is a bijection between covered maps with $n$ edges and code-shuffles of length $2n$.  Moreover, if $w$ is the code of the covered map $(M,S)$, then $w_{|A}$ is the code of the unicellular map $M_{|S}$ (on the alphabet $A_{|S|/2}$) and $w_{|B}$ is the code of the dual unicellular map $M^*_{|\bS}$ (on the alphabet $B_{|\bS|/2}$).
\end{prop}

\begin{proof} To see that $\phi$ is injective, observe first that the code-shuffle allows to recover the root of the map $M=(H,\si,\al)$, the subset $S=A_k$, the edge-permutation $\alpha$ and the motion function $\theta=(w_1,\ldots ,w_{2n})$. From this, the vertex-permutation $\si$ is deduced by $\si(h)=\theta\al(h)$ if $h\in S$ and $\si(h)=\theta(h)$ otherwise. We now prove that $\phi$ is surjective. For this, it is sufficient to prove that starting from any shuffle-code, the pair $(M,S)$ defined as above is a covered map. First note that the permutations $\si$ and $\al$ clearly act transitively on $H$ since $\theta$ is cyclic, hence $M$ is a map. Now, the fact that $(M,S)$ is a a covered map is a consequence of  Lemma~\ref{prop:covered=cyclic} since $\theta$ is the motion function of $(M,S)$ and is cyclic.

We now prove the second statement. Let $w_{A}=w_1',\ldots,w'_{2k}$ and  $w_{B}=w_1'',\ldots,w''_{2l}$. By definition of restrictions, $\theta_{|S}=(w_1',\ldots,w'_{2k})$ and $\theta_{|\bS}=(w_1'',\ldots,w''_{2l})$. Moreover, by Proposition~\ref{prop:covered=cyclic}, these restrictions $\theta_{|S}$ and $\theta_{|\bS}$ correspond respectively to the face-permutations of $M_{|S}$ and $M^*_{|\bS}$. Recall also that the root $r_1$ of $M_{|S}$ is $\si^{i}(r)$, where $r$ is the root of $M$ and $i$ is the least integer such that $\si^{i}(r)\in S$. Equivalently, $r_1=\theta^{i}(r)$ where $i$ is the least integer such that $\theta^{i}(r)\in S$, hence $r_1=w_1'$. Similarly, the root $r_2$ of $M^*_{|\bS}$ is $\phi^{j}(r)$ where $j$ is the least integer such that $\phi^{j}(r)\in \bS$, or equivalently  $r_2=\theta^{j}(r)$ where $j$ is the least integers such that $\theta^{j}(r)\in \bS$, hence $r_1=w_1''$. Thus, the words $w_{|A}$ and $w_{|B}$ are the codes of the unicellular maps $M_{|S}$ and $M^*_{|\bS}$ respectively.
\end{proof}

We now explore the enumerative consequence of Proposition~\ref{prop:covered=shuffle}. Let $A_{g}(n)$ be the number of unicellular maps of genus $g$ with  $n$ edges. Let $C_{g_1,g_2}(n_1,n_2)$ (resp. $C_{g_1,g_2}(n)$) be the number of covered maps $(M,S)$ such that the unicellular maps $M_{|S}$ and $M^*_{|\bS}$ have respectively $n_1$ and $n_2$ edges (resp. a total of $n$ edges) and genus $g_1$ and $g_2$. Since there are ${2n_1+2n_2 \choose 2n_1}$ ways of shuffling unicellular codes of length $2n_1$ and $2n_2$, Proposition~\ref{prop:covered=shuffle} gives 
\begin{eqnarray}\label{eq:coveredbsasic}
C_{g_1,g_2}(n_1,n_2) = {2n_1+2n_2 \choose 2n_1} A_{g_1}(n_1) A_{g_2}(n_2),
\end{eqnarray}
and
\begin{eqnarray}\label{eq:covered=shuffle1}
C_{g_1,g_2}(n) = \sum_{m=0}^{n}{2n \choose 2m} A_{g_1}(m) A_{g_2}(n-m).
\end{eqnarray}
An alternative equation (used in Section~\ref{sec:enumeration}) is obtained by fixing the number of vertices of $M_{|S}$ and $M^*_{|\bS}$ instead of their genus. Let $A^v(n)$ be the number of unicellular maps with $v$ vertices and $n$ edges ($A^v(n)=A_{(n-v+1)/2}(n)$ by Euler relation and this number is 0 if $n-v+1$ is odd). Let also $C^{v,f}(n)$ be the number of covered maps with $v$ vertices, $f$ faces and $n$ edges (hence with genus $g=(n-v-f+2)/2$). Proposition~\ref{prop:covered=shuffle} gives  
\begin{eqnarray}\label{eq:covered=shuffle2}
C^{v,f}(n) = \sum_{m=0}^{n}{2n \choose 2m} A^v(m) A^f(n-m).
\end{eqnarray}

Equation~\eqref{eq:covered=shuffle1} generalizes the results used by Mullin~\cite{Mullin:tree-rooted-maps} and by Lehman and Walsh~\cite{Walsh:counting-maps-2} in order to count tree-rooted maps. Indeed, the number of tree-rooted maps of genus $g$ with $n$ edges is 
\begin{eqnarray}\label{eq:tree-rooted=shuffle}
T_{g}(n) = C_{0,g}(n) = \sum_{n=0}^{m}{2n \choose 2m} \Cat(m) A_{g}(n-m),
\end{eqnarray}
where $\Cat(m)=\frac{1}{m+1}{2m \choose m}$ is the $m$th Catalan number. In~\cite{Mullin:tree-rooted-maps},  Mullin proved  Equation~\eqref{eq:Mullin} by applying the Chu-Vandermonde identity to~\eqref{eq:tree-rooted=shuffle} (in the case $g=0$). Similarly, in~\cite{Walsh:counting-maps-2}, Lehman and Walsh proved  Equation~\eqref{eq:Lehman} by applying the Chu-Vandermonde identity to~\eqref{eq:tree-rooted=shuffle} (in the case $g=1$).
In~\cite{BeCaRo:asymptotic-nb-tree-rooted}, Bender \emph{et al.} used the asymptotic formula 
$$A_g(n)\sim_{n\to\infty} \frac{n^{3g-\frac{3}{2}}}{12^gg!\sqrt{\pi}}4^n \left(1+O\left(\frac{1}{\sqrt{n}}\right)\right)$$
which they derived from the expressions given in~\cite{Walsh:counting-maps-1}, together with~\eqref{eq:tree-rooted=shuffle} in order to determine the asymptotic number of tree-rooted maps of genus $g$ and obtained: 
$$T_{g}(n)\sim_{n\to\infty} \frac{4}{\pi g!96^g}n^{3g-3}16^n.$$

Applying the same techniques as Bender \emph{et al.} to Equation~\eqref{eq:covered=shuffle1} gives the asymptotic number of covered maps: 
\begin{eqnarray}
C_{g_1,g_2}(n)\sim_{n\to\infty} {g_1+g_2 \choose g_1}\frac{4}{\pi g!96^g}n^{3g-3}16^n.
\end{eqnarray}
In particular, the the total number of covered maps of genus $g$ with $n$ edges satisfies:
\begin{eqnarray}\label{eq:asymptotic-covered}
C_{g}(n)=\sum_{h=0}^g C_{h,g-h}(n) \sim \frac{4}{\pi g!48^g}n^{3g-3}16^n.
\end{eqnarray}
Hence the proportion of tree-rooted maps among covered maps of genus $g$ tends to $1/2^{g}$ when the size $n$ goes to infinity. We have no simple combinatorial interpretation of this fact.\\

This concludes our preliminary exploration of covered maps. We now leave the world of shuffles and concentrate on the main subject of this paper, that is, the bijection $\Psi$ between covered maps and pairs made of a tree and a unicellular bipartite map.\\


\section{The bijection.}\label{sec:bijection}
This section contains our main result, that is, the description of the bijection $\Psi$ between covered maps and pairs made of a tree and a bipartite unicellular map called the \emph{mobile}.  

Let $(M,S)$ be a covered map. The bijection $\Psi$ consists of two steps. At the first step, the submap $S$ is used to define an orientation $(I,O)$ of $M$; see Figure~\ref{fig:bijection1}.   At the second step of the bijection, which we call \emph{unfolding}, the vertices of the map $M$ are split into several vertices (the rule for the splitting is given in terms of the orientation $(I,O)$;  see Figure~\ref{fig:split}).
The map obtained after these splits is a plane tree  $A$, and the information about the splitting process can be encoded into a bipartite unicellular map $B$. The tree $A=\Psi_1(M,S)$ and the mobile $B=\Psi_2(M,S)$ are represented in Figure~\ref{fig:Bijection3}.
We now describe the two steps of the bijection $\Psi$ in more details.\\ 

\titre{Step 1: Orientation $\Delta$.} 
The orientation step is represented in Figure~\ref{fig:bijection1}. One starts with a covered map $(M,S)$ and obtains an oriented map $(M,(I,O))$. 
Topologically, the orientation $(I,O)$ is obtained by turning around the submap $\MS$ (in counterclockwise direction starting from the root) and orient each edge of $M$ according to the following rule: 
\begin{itemize}
\item each edge in $\MS$ is oriented in the direction it is followed for the first time during the tour,
\item each edge not in $\MS$ is oriented in such a way that the ingoing half-edge is crossed before the outgoing half-edge during the tour.
\end{itemize}

Let us now make definitions precise in terms of the combinatorial definition of maps. Let $(M,S)$ be a covered map, let $r$ be its root, and let $\theta$ be its motion function. Recall from Proposition~\ref{prop:covered=cyclic} that $\theta$ is a cyclic permutation on the set $H$ of half-edges. Therefore, one obtains a total order $\lS$, named \emph{appearance order}, on the set $H$ by setting $r\lS \theta(r)\lS\cdots\lS\theta^{|H|-1}(r)$. Topologically, the appearance order is the order in which half-edges of $M$ appear when turning around the spanning submap $T=M_{|S}$ in counterclockwise order starting from the root. For instance, the order obtained for the spanning submap $T$ in Figure~\ref{fig:bijection1}(a) is $a_1\lS b_1 \lS a_2 \lS b_2 \lS \ba_2\lS b_3\lS \cdots \lS \ba_3$. 
We now define the oriented map $(M,(I,O))=\Delta(M,S)$  which is represented in Figure~\ref{fig:bijection1}(b).  

\begin{definition} 
Let $(M,S)$ be a covered map with half-edge set $H$. The mapping $\Delta$ associates to $(M,S)$ the oriented map  $(M,(I,O))$, where the set $I$ of ingoing half-edges contains the half-edges $h\in S$ such that $\alpha(h)\lS h$ and the half-edges $h\notin S$ such that  $h\lS \alpha(h)$ (and $O=H\setminus I$).
\end{definition}

We now characterize the image of the mapping $\Delta$ by defining left-connected orientations.
Let $M=(H,\si,\al)$ be a map and let $(I,O)$ be an orientation. 
Let $h_0$ denote the root of $M$. A \emph{left-path} is a sequence  $h_1,h_2,\ldots,h_{k}$ of ingoing half-edges such that for all $i=1\ldots k$, there exists an integer $q_i>0$ such that $h_{i-1}=\si^{q_i}(\al(h_{i}))$ and $\si^{p}(\al(h_{i}))\in O$ for all $p=0\ldots q_i-1$. In words, a left-path is a directed path starting from the arrow pointing the root-corner and such that no ingoing half-edges is incident to the left of the path. Clearly, for any ingoing half-edge $h$, there exists at most one left-path $h_1,h_2,\ldots,h_{k}$ whose extremity is $h_k$ is $h$.  We say that an oriented map $(M,(I,O))$ is \emph{left-connected} if every ingoing half-edge is the extremity of a left-path.

\begin{figure}[ht!]\begin{center} \input{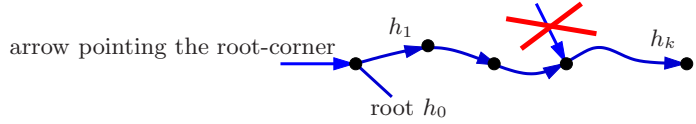}\caption{A left-path.}\label{fig:left-path} \end{center}\end{figure}

\begin{thm}\label{thm:Delta}
The mapping $\Delta$ is a bijection between covered maps and left-connected maps.
\end{thm}

 The proof of Theorem~\ref{thm:Delta} is postponed to Section~\ref{sec:proofs}.\\

\titre{Remark on the planar case:} It is shown in~\cite[Prop. 3]{OB:boisees} that the mapping $\Delta$ is a bijection between planar covered maps (i.e. tree-rooted maps) and planar oriented maps which are \emph{accessible} (any vertex can be reach from the root-vertex by a directed path) and \emph{minimal} (no directed cycle has the root-face on its right). Thus, in the planar case the left-connected orientations are the accessible minimal orientations.\\

\titre{Step 2: Unfolding $\Lambda$.}  
The unfolding step is represented in Figures~\ref{fig:bijection2} and~\ref{fig:Bijection3}. One starts with a left-connected map $(M,(I,O))$ and obtains two maps $A=\Lambda_1(M,(I,O))$ and $B=\Lambda_2(M,(I,O))$.  The map  $A$ is a plane tree and the map  $B$ is a bipartite unicellular map (with \emph{black} and \emph{white} vertices). By analogy with the paper~\cite{BDFG:mobiles}, we call the bipartite unicellular $B$ the \emph{mobile} associated with the left-connected map $(M,(I,O))$.
Let us start with the topological description of this step. Let $v$ be a vertex of the oriented map $(M,(I,O))$ and  let $h_1,\ldots,h_d$ be the incident half-edges in counterclockwise order around~$v$ (here it is important to consider the arrow pointing the root-corner as an ingoing half-edge). If the vertex $v$ is incident to $k>0$ ingoing half-edges, say $h_{i_1},h_{i_2},\ldots,h_{i_k}=h_d$, then the vertex $v$ of $M$ will be split into $k$ vertices $v_1, v_2,\ldots,v_k$ of the tree $A$. The splitting rule is represented in Figure~\ref{fig:split}: for $j=1\ldots k$, the vertex $v_j$ of the trees $A$ is incident to the half-edges $h_{i_{j-1}+1},h_{i_{j-1}+2},\ldots,h_{i_j}$.\\ 

\begin{figure}[ht!]\begin{center} \input{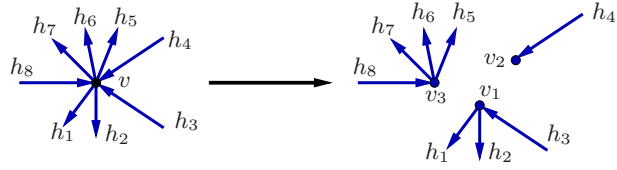}\caption{Splitting of a vertex $v$ incident to 3 ingoing half-edges $h_3,h_4,h_8$.}\label{fig:split} \end{center}\end{figure}

Observe that the splitting of the vertex $v$ can be written conveniently in terms of permutations. Indeed, seeing the vertex $v$ as the cycle  $(h_1,\ldots,h_d)$ of the vertex-permutation~$\si$ and the vertices $v_1=(h_1,\ldots h_{i_1})$, \ldots, $v_k=(h_{i_{k-1}+1},\ldots,h_{i_k})$ as cycles  of the vertex-permutation $\tau$ of the tree $A$ gives the following relation between $v$ and the product of cycles $v'=v_1v_2\ldots,v_k$ (these are both permutations on $\{h_1,\ldots,h_k\}$) 
$$v=v'\pi_\circ,$$ 
where $\pi_\circ$ is the permutation such that $\pi_\circ(h)=h$ if $h\in O$ and $\pi_\circ(h_{i_j})=h_{i_{j+1}}$ for $j=1,\ldots,l$. Hence, $v'=v\pi_\circ^{-1}$, where $\pi_\circ=v_{|I}$ (with the convention that the restriction $v_{|I}$ acts as the identity on $O$). The cycle $(h_{i_1},h_{i_2},\ldots,h_{i_l})$ of $\pi_\circ$ will represent one of the (white) vertices of the bipartite unicellular map $B$. This white vertex is represented in Figure~\ref{fig:bijection2}(a).\\

\begin{figure}[t!]
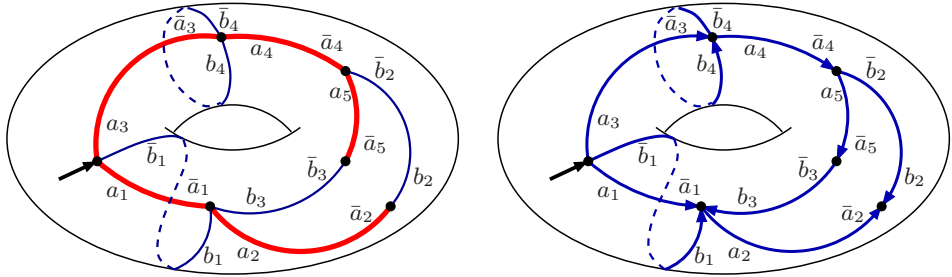

\begin{center} 
\input{Bijection.pstex_t}\hspace{.5cm}\input{Bijection1.pstex_t}
\caption{(a) A covered map of genus 1 (the unicellular submap is indicated by thick lines) (b) The associated oriented map.}\label{fig:bijection1} 
\end{center}
\end{figure}
\begin{figure}[t!]
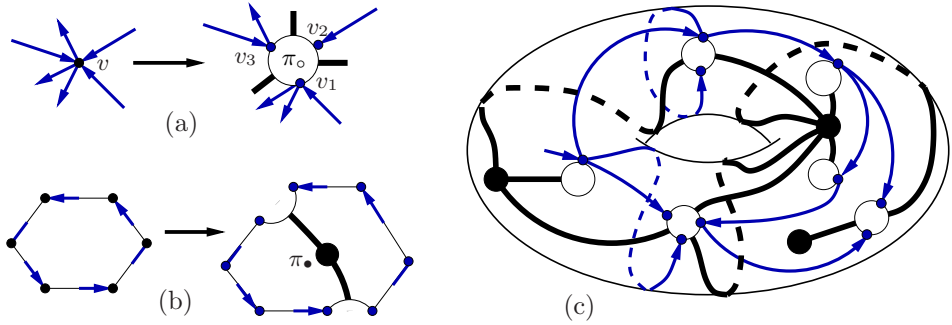

\begin{center} 
\input{Unfolding.pstex_t}\hspace{.7cm}\input{bijection2.pstex_t}
\caption{Representation of the unfolding: (a) around one vertex (this defines one cycle of $\pi_{\circ}$);  (b) around one face (this defines one cycle of $\pi_{\bu}$); (c) on the map of Figure~\ref{fig:bijection1}.}\label{fig:bijection2} 
\end{center}
\end{figure}
\begin{figure}[t!]
\begin{center} 
\input{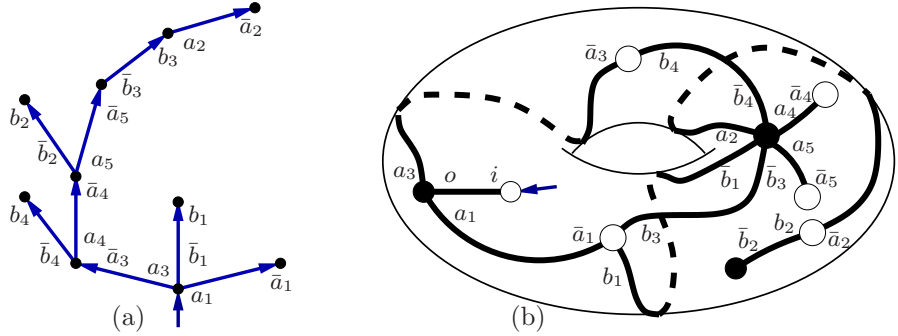}
\caption{The image by $\Psi$ of the covered map $(M,S)$ of Figure~\ref{fig:bijection1}. (a) The tree $\Psi_1(M,S)$. (b) The unicellular map $\Psi_2(M,S)$.}\label{fig:Bijection3} 
\end{center}
\end{figure}
We now describe the unfolding step in more details. Let $r$ be the root of the map $M=(H,\si,\al)$ and let $\phi=\si\al$ be its face-permutation .  We consider two new \emph{half-edges} $i$ and $o$ not in $H$ and define $H'=H\cup\{i,o\}$, $I'=I\cup \{i\}$ and $O'=O\cup \{o\}$ (the half-edge $i$ should be thought as this half-edge pointing to the root-corner, while $o$ should be thought as its dual). We define the involution $\alpha'$ on $H'$ by setting $\alpha'(i)=o$ and $\alpha'(h)=\alpha(h)$ for all $h\in H$. 
We also define $\si'$ as the permutation on $H'$ obtained  from $\si$ by inserting the new half-edge $i$ just before the root~$r$ in the cycle of $\si$ containing $r$ and creating a cycle made of $o$ alone (that is, $\si'(o)=o$). 
Similarly we define $\phi'$ as the permutation on $H'$ obtained  from $\phi$ by inserting the new half-edge~$o$ just before~$r$ in the cycle of $\phi$ containing $r$ and creating a cycle made of $i$ alone. Recall that $\phi=\si\al$ and observe that $\phi'=(i,o)\si'\al'$. We consider the restrictions 
\begin{eqnarray}\label{eq:pibullet}
\pi_\circ=\si'_{|I'}~~\textrm{ and }~~\pi_\bu=\phi'_{|O'}~.
\end{eqnarray}
In the example of Figure~\ref{fig:bijection1}, one gets $\pi_\circ=(i)(\ba_1,b_1,b_3)(\ba_2,b_2)(\ba_3,b_4)(\ba_5)(\ba_4)$ and $\pi_\bu=(o,a_1,a_3)(\bb_1,a_2,\bb_4,a_4,a_5,\bb_3)(\bb_2)$. 
We now define the permutations $\pi$ and $\tau'$ on $H'$, and a permutation $\tau$ on $H$ by setting 
\begin{eqnarray}\label{eq:pitau}
\pi=\pi_\circ \pi_\bu^{-1},~~ \tau'=\si'\pi_\circ^{-1} ~\textrm{ and }~ \tau=\tau'_{|H}~,
\end{eqnarray}
where a slight abuse of notation is done by considering that $\pi_\circ=\si'_{|I'}$ acts as the identity on $O'$ and that $\pi_\bu=\phi'_{|O'}$ acts as the identity on~$I'$. It is easily seen that $\tau'(o)=o$. On the other hand, we will show (Lemma~\ref{lem:not-alone}) that the half-edge $i$ is not alone in its cycle of $\tau'$. Hence, the half-edge $t=\tau'(i)$ is distinct from $i$ and $o$. We now consider the pseudo maps $A=(H,\tau,\al)$ with root $t=\tau'(i)$  and $B=(H',\pi,\al')$ with root $i$.

\begin{definition}
We denote by $\Lambda$ the mapping which to a left-connected map $(M,(I,O))$ associates the pair $(A,B)$. We also denote $\Psi=\Lambda\circ \Delta$. Lastly if $(M,S)$ denotes the covered map such that $(M,(I,O))=\Delta(M,S)$, we denote $\Psi_1(M,S)= \Lambda_1(M,(I,O))=A$ and  $\Psi_2(M,S)=\Lambda_2(M,(I,O))=B$.
\end{definition}

The images $(A,B)$ of the covered map in Figure~\ref{fig:bijection1}(a) by the mappings $\Psi_1$ and $\Psi_2$ are represented respectively in Figure~\ref{fig:Bijection3}(a) and~(b). In terms of permutations, one gets $A=(H,\tau,\al)$ and $B=(H',\pi,\al')$, where 
$$\tau=(a_1,\bb_1,\ba_3)(\ba_1)(b_1)(\ba_3,a_4,\bb_4)(\ba_4,a_5,b_2)(\ba_5,\bb_3)(b_3,a_2)(\ba_2)(\bb_2)(b_4)$$ and 
$$\pi=(i)(\ba_1,b_1,b_3)(\ba_2,b_2)(\ba_3,b_4)(\ba_4)(o,a_3,a_1)(\bb_3,a_5,a_4,\bb_4,a_2,\bb_1)(\bb_2).$$ 

Our main result is the following theorem which will be proved in Section~\ref{sec:proofs}.

\begin{thm}\label{thm:bij}
The mapping $\Psi=\Lambda\circ\Delta$ which to a covered map $(M,S)$ associates the pair $(\Psi_1(M,S),\Psi_2(M,S))$ is a bijection between covered maps of size $n$ and genus $g$ and pairs made of a plane tree $\Psi_1(M,S)$ of size $n$ and a bipartite unicellular map $\Psi_2(M,S)$ of size $n\!+\!1$ and genus~$g$.  Moreover by coloring the vertices of the bipartite map $\Psi_2(M,S)$ in two colors, say white and black, with the root-vertex being white, one gets $v(M)$ white vertices and $f(M)$ black vertices.
\end{thm}

\titre{Remark (topological intuition).} From Figures~\ref{fig:bijection2}(a) and (b), the reader should see that the mobile $B=\Lambda_2(M,(I,O))$ has white vertices (the cycles of $\pi_\circ$ made of half-edges in $I'$) corresponding to the vertices of $M$ and  black vertices (the cycles of $\pi_\bu^{-1}$ made of half-edges in $O'$) corresponding to the faces of $M$. The topological intuition explaining that $A=\Psi_1(M,S)$ is connected, is that left-paths are preserved during the unfolding step. It then follows that $A$ is a plane tree because its number of vertices is one more than its number of edges. The topological intuition  explaining that the mobile $B$ is a unicellular map comes from the fact that $A$ can reach every white corners of $B$ (without crossing its edges). Indeed, this implies that the pseudo-map $B$ has no contractible cycles. 
Using this fact and the relation between the number of vertices and edges shows that $B$ is a unicellular map of genus $g(M)$.\\


\section{Enumerative corollaries} \label{sec:enumeration}
In this section we explore the enumerative consequences of the bijection $\Psi$, and establish the equivalence between the formula \eqref{eq:HarerZagier} of Harer and Zagier and the formula \eqref{eq:JacksonAdrianov} of Jackson.

Recall the notations of Section~\ref{sec:shuffles}:  $A_g(n)$, $B_g(n)$, $C_g(n)$ are respectively the number of general unicellular maps, bipartite unicellular maps, and covered maps with $n$ edges and genus $g$. Similarly $A^v(n)$, $B^{v,f}(n)$, $C^{v,f}(n)$ are respectively the number of general unicellular maps with $v$ vertices, bipartite unicellular maps with $v$ white and $f$ black vertices, and covered maps with $v$ vertices and $f$ faces having $n$ edges. The first direct consequence of Theorem~\ref{thm:bij} is:
\begin{theorem}
\label{thm:allcovered}
The numbers $C_g(n)$ and $C^{v,f}(n)$ of covered maps, and the numbers $B_g(n)$ and $B^{v,f}(n)$ of bipartite maps are related by:
\begin{eqnarray}
C_g(n) &=& \Cat(n)\, B_g(n+1)\label{eq:covered}\\
C^{v,f}(n) &=& \Cat(n)\, B^{v,f}(n+1)\label{eq:allcovered}
\end{eqnarray}
where $\Cat(n)=\frac{1}{n+1}{2n\choose n}$.
\end{theorem}

Observe that for $g=0$ we have $B_0(n+1)=\Cat(n+1)$, so that Equation~\eqref{eq:covered} coincides with Mullin's formula (Equation~\eqref{eq:Mullin}).
Using known closed-form expressions for the numbers $B_g(n)$ (see~\cite{Goupil-Schaeffer:n-cycles}), one obtains similar formulas for other genera.
For example in genus $1$ and $2$ we obtain the following expressions for the number of covered maps with $n$ edges:
$$C_1(n) = \Cat(n)\frac{(2n-1)!}{6(n-2)!(n-1)!},~~~~~~~
C_2(n)=\Cat(n)\frac{(5n^2+3n+4)(2n)!}{1440(n-1)!(n-4)!}.$$

We now examine the special case of the torus (genus 1). By Lemma~\ref{cor:dual-covered}, a covered map on the torus is either a tree-rooted map (the submap has genus 0, that is, is a spanning tree) or the dual of a tree-rooted map.  Since duality is involutive, \emph{exactly half} of toroidal covered maps of given size are tree-rooted maps. This gives the first bijective proof to the following result:
\begin{corollary}[Lehman and Walsh~\cite{Walsh:counting-maps-2}]
The number of tree-rooted maps with $n$ edges on the torus is:
\begin{eqnarray*}\label{eq:LW}
T_{1}(n)=\frac{1}{2} C_1(n) = \frac{(2n)!(2n-1)!}{12(n+1)!n!(n-1)!(n-2)!}.
\end{eqnarray*}
\end{corollary}

Another enumerative byproduct of our bijection is a relation between the numbers of general and bipartite unicellular maps.
Indeed, by comparing the expression of $C^{v,f}(n)$ obtained by the shuffle approach (Equation~\eqref{eq:covered=shuffle2}) with the one of Theorem~\ref{thm:allcovered}, we obtain the following:
\begin{theorem}
The numbers of bipartite and monochromatic unicellular maps are related by the
formula:
\begin{eqnarray}
\label{eq:bipartitefrommono}
B^{v,f}(n+1) = \sum_{n_1+n_2=n} \frac{n!(n+1)!}{(2n_1)!(2n_2)!} A^v(n_1)
A^f(n_2)
\end{eqnarray}
Moreover, this formula establishes the equivalence between the formula~\eqref{eq:HarerZagier} due to Harer and Zagier and  the formula~\eqref{eq:JacksonAdrianov} due to Jackson.
\end{theorem}
\begin{proof}
The first statement is obtained by comparing  Equations~\eqref{eq:covered=shuffle2} and~\eqref{eq:allcovered}.
We now show that~\eqref{eq:HarerZagier} implies~\eqref{eq:JacksonAdrianov}. Using \eqref{eq:HarerZagier} one gets:
\begin{eqnarray*}
& &\sum_{p,q\geq 1} B^{p,q}(n+1) y^p z^q \\
&\stackrel{Eq.~\eqref{eq:bipartitefrommono}}{=}&
\sum_{p,q\geq 1} \sum_{n_1+n_2=n}
\frac{n!(n+1)!}{(2n_1)!(2n_2)!} A^p(n_1)
A^q(n_2)y^pz^q \\
&=&
\sum_{n_1+n_2=n}
\frac{n!(n+1)!}{(2n_1)!(2n_2)!}
\left( \sum_{p\geq 1} A^p(n_1) y^p  \right)
\left( \sum_{q\geq 1} A^q(n_2) z^q  \right) \\
&\stackrel{Eq.~\eqref{eq:HarerZagier}}{=}&
\sum_{n_1+n_2=n}
\frac{n!(n+1)!}{2^nn_1!n_2!}
\sum_{i,j \geq 1} 2^{i+j-2}
{n_1 \choose i-1}
{n_2 \choose j-1}
{y \choose i}
{z \choose j} \\
&=&
\sum_{i,j \geq 1}
\frac{2^{i+j-n-2}n!(n+1)!}{(i-1)!(j-1)!}{y \choose i}{z \choose j}
\!\sum_{n_1+n_2=n\atop n_1\geq i-1,\,n_2\geq j-1}\frac{1}{(n_1-i+1)!(n_2-j+1)!}\\
\end{eqnarray*}
where the second and fourth equalities just correspond to  rearrangements of the summations. Moreover, the inner sum in the last equation is equal to $2^{n-i-j-2}/(n-i-j+2)!$ by Newton's binomial theorem. This gives Jackson's formula~\eqref{eq:JacksonAdrianov}.\\ 
Conversely, observe that the numbers $A^p(n)$ are uniquely determined by the value of the sums $\sum_{n_1+n_2=n}
\frac{n!(n+1)!}{(2n_1)!(2n_2)!} A^p(n_1)A^p(n_2)$ for all $p,n$ (by induction on $n$). Thus, the calculations above show that Jackson's formula~\eqref{eq:JacksonAdrianov} implies  the Harer-Zagier formula~\eqref{eq:HarerZagier}.
\end{proof}

\titre{Remark.} Connoisseurs know that formulas \eqref{eq:HarerZagier} and \eqref{eq:JacksonAdrianov} can be interpreted in terms of unicellular maps with colored vertices (see e.g.~\cite[sec. 3.2.7]{Lando-zvonkin}). For those readers, we point out that the equivalence between Harer-Zagier and Jackson's formulas can be seen in terms of colorings as well:  $\Psi$ is a bijection between shuffles of two unicellular maps with vertices colored using all colors respectively in $\{1,\ldots,i\}$ and $\{1,\ldots,j\}$ and pairs made of a plane tree and a unicellular map with black and white vertices colored using all colors respectively in $\{1,\ldots,i\}$ and $\{1,\ldots,j\}$.\\


\section{Proofs and inverse bijection.} \label{sec:proofs}
This section is devoted to the proof of Theorems~\ref{thm:Delta} and~\ref{thm:bij} concerning respectively the orientation and unfolding steps of the bijection~$\Psi$.

\subsection{Proofs concerning the orientation step.}
In this Subsection, we prove Theorem~\ref{thm:Delta} about the orientation step $\Delta$ and and define the inverse mapping $\Gamma$. Given an oriented map $(M,(I,O))$ with vertex-permutation $\si$ and face-permutation $\phi$, the \emph{backward function} $\beta$ is defined by $\beta(h)=\si(h)$ if $h\in O$ and $\beta(h)=\phi(h)$ otherwise.  We point out that the backward function is not a permutation, since $\beta(h)=\beta(\al(h))$ for any half-edge $h$.

\begin{lemma}\label{lem:equiv-left-connected}
Let  $(M,(I,O))$ be an oriented map with root $h_0$ and let $\be$ be the backward function. The oriented map $(M,(I,O))$ is left-connected if and only if for any half-edge, there exists an integer $q>0$ such that  $\beta^{q}(h)=h_0$.
\end{lemma}

\begin{proof} 
Suppose first that $(M,(I,O))$ is left-connected. Let $h$ be a half-edge. If $h$ is ingoing, then it is the extremity of a left path  $h_1,\ldots,h_k=h$. By definition of left-paths, there exist positive integers $q_1,\ldots,q_k$ such that $h_{i-1}=\beta^{q_i}(h_i)$ for all $i=1\ldots k$. Hence, $\be^q(h)=h_0$ for $q=q_1+\cdots+q_k$. Now, if $h$ is outgoing, $\be(h)=\beta(\al(h))$, hence there exists $q>0$ such that $h_0=\be^q(\al(h))=\be^q(h)$.

Suppose conversely that for any half-edge $h$, there exists an integer $q>0$ such that  $\beta^{q}(h)=h_0$. In this case, for any ingoing half-edge $h$, the sequence $h_1,h_2,\ldots,h_k=h$ of ingoing half-edges appearing (in this order) in the sequence $\beta^{q-1}(h),\beta^{q-2}(h),\ldots,\beta(h),h$ is a left-path. Hence $(M,(I,O))$ is left-connected.
\end{proof}

\begin{proposition}\label{prop:returnsleftconnected}
The image of any covered map by the mapping $\Delta$ is left-connected.
\end{proposition}

\begin{proof} Let $(M,S)$ be a covered map, where the map $M=(H,\si,\al)$ has root $r$, and  let $(M,(I,O))$ be its image by $\Delta$. Our strategy is to prove that for any half-edge $h$ such that  $\beta(h)\neq r$, one has $h\lS\beta(h)$. This will clearly prove that the sequence $\beta(h), \beta^2(h), \beta^3(h)\ldots$ must contain the root $r$, and by Lemma~\ref{lem:equiv-left-connected}, that $(M,(I,O))$ left-connected. 

We distinguish four cases, depending on the fact that $h$ is in $I$ or $O$, and in $S$ or $\bS$. In these four cases, we denote $h'$ the half-edge $\theta^{-1}(\beta(h))$, where $\theta$ is the motion function of $(M,S)$. Observe that $h'\lS \be(h)$ since by hypothesis $\be(h)\neq r$.\\
{\bf{Case 1:}} $h$ is in $O$ and in $\bS$. In this case, one has $\beta(h)=\si(h)=\theta(h)$, hence $h'=h$. Thus $h=h'\lS \beta(h)$.\\
{\bf{Case 2:}} $h$ is in $O$ and in $S$. In this case,  one has $\beta(h)=\si(h)=\theta(\al(h))$, hence $h'=\al(h)$. Moreover, by definition of  $\Delta$,  one has $h\lS \al(h)$ thus $h\lS h' \lS \beta(h)$.\\
{\bf{Case 3:}} $h$ is in $I$ and in $\bS$. In this case,  one has $\beta(h)=\si(\al(h))=\theta(\al(h))$, hence $h'=\al(h)$. Moreover, by definition of  $\Delta$,  one has $h\lS \al(h)$, thus $h\lS h' \lS \beta(h)$.\\
{\bf{Case 4:}} $h$ is in $I$ and in $S$. In this case, one has $\beta(h)=\si\al(h)=\theta(h)$, hence $h'=h$. Thus, $h=h'\lS \beta(h)$. 
\end{proof}

We will now define a mapping $\Gamma$ that we will prove to be the inverse of $\Delta$. Let us first give the intuition behind the injectivity of $\Delta$ by considering a covered map $(M,S)$ with motion function $\theta$ and its image $(M,(I,O))$ by $\Delta$. Observe from the definition of $\Delta$ that the root $r$ of $M$ is in $S$ if and only if it is in $O$. Thus, it is possible to know from the orientation $(I,O)$ whether $r$ belongs to $S$ or not, and thereby deduce the next half-edge $h_1=\theta(r)$ around the submap $\MS$. The same reasoning will allow to determine, from the orientation $(I,O)$, whether the half-edge $h_1$ belongs to $S$ and deduce the next half-edge $h_2=\theta(h_1)$ around $\MS$ and so on\ldots This should convince the reader that the mapping $\Delta$ is injective and highlight the definition of $\Gamma$ given below.\\
 
It is convenient to define $\Gamma$ as a procedure which given an oriented map $(M,(I,O))$ with root $r$ returns a subset $S$ of half-edges. This procedure visit some half-edges of $M$ starting from the root $r$, and decide at each step whether the current half-edge $h$ belongs to the set $S$ or not. 

\begin{definition}\label{def:Gamma}
The mapping $\Gamma$ associates to an oriented map $(M,(I,O))$ the spanned map obtained by the following procedure.\\
\titre{Initialization:} Set $S=\emptyset$, $R=\emptyset$ and set the current half-edge $h$ to be the root $r$.\\
\titre{Core:}\\ 
\indent$\bullet$  If $h\notin S\cup R$ do:\\
\indent \indent If $h$ is in $O$ then add $h$ and $\al(h)$ to $S$; otherwise add $h$ and $\al(h)$ to $R$.\\
\indent$\bullet$ Set the the current half-edge $h$ to be $\si\al(h)$ if $h$ in $S$ and $\si(h)$ otherwise. \\  
Repeat until the current half-edge $h$ returns to be $r$.\\
\titre{End:} Return the spanned map $(M,S)$.
\end{definition}

We first prove the termination of the procedure $\Gamma$.

\begin{lemma}
\label{lemma:defGamma}
For any oriented map $(M,(I,O))$, the procedure $\Gamma$ terminates and returns a spanned map $(M,S)$.
Moreover, the list of all successive current half-edges visited by the procedure is the cycle containing the root $r$ of the motion function~$\theta$ associated to $(M,S)$.
\end{lemma}

\begin{proof}
We first prove that the procedure $\Gamma$ terminates. Observe that, at any step of the procedure, the sets $S$ and $R$ are disjoint and stable by $\alpha$. Moreover, these sets are both increasing, hence they are constant after a while, equal to some sets $S_\infty$ and $R_\infty$ which are disjoint. Let  $\theta_\infty$ be the motion function of the submap $M_{|S_\infty}$.  
Then, at each core step of the procedure, the current half-edge $h$ becomes the half-edge $\theta_\infty(h)$. Indeed, if at the current step $h$ is in $S$, then $h$ is in $S_\infty$ (since the set $S$ cannot decrease) so that $\theta_\infty(h)=\sigma\alpha(h)$, while if $h$ is in $R$, then $h$ is in $R_\infty$ (since the set $R$ cannot decrease), hence it is not in $S_\infty$ so that $\theta_\infty(h)=\sigma(h)$. 

Hence the sequence of all successive current half-edges form a cycle of the \emph{permutation} $\theta_\infty$. Since the procedure starts with $h$ equal to the root $r$, it follows that $r$ is reached a second time, and that the procedure terminates. Finally, the spanned map returned by the algorithm is $(M,S_\infty)$, which concludes the proof.
\end{proof}

\begin{proposition} \label{prop:returnscovered}
The image of a  left-connected map  by the mapping $\Gamma$ is a covered map.
\end{proposition}

\begin{proof}  Let $(M,(I,O))$ be a left-connected map.
We denote by $r$ the root of $M=(H,\si,\al)$, by $\beta$ the backward function of $(M,(I,O))$, and by $\theta$ the motion function of $(M,S)=\Gamma((M,(I,O)))$. Let $K$ be the set of half-edges contained in the cycle of the motion function $\theta$ containing the root $r$. In order to prove the proposition, it suffices to prove that $K=H$. Indeed, in this case the motion function $\theta$ is cyclic which implies that $(M,S)$ is a covered map by Proposition~\ref{prop:covered=cyclic}. 

Moreover, since  $(M,(I,O))$ is left-connected, for any half-edge $h$ in $H$, there exists a positive integer $q$ such that $\beta^q(h)$ is the root $r$ which belongs to $K$. Hence, it suffices to prove that \emph{any half-edge $h$ such that $\beta(h)$ is in $K$, is also in $K$}.

Let $h$ be an half-edge such that $\be(h)$ is in  $K$ and let $h'=\theta^{-1}(\be(h))$. Observe that, by definition of $K$, the half-edge $h'$ is in $K$. We now distinguish four cases, depending on the fact that $h$ is in $I$ or $O$, and in $S$ or $\bS$. \\
{\bf{Case 1:}} $h$ is in $O$ and in $\bS$. In this case, one has $\beta(h)=\si(h)=\theta(h)$, hence $h'=h$. Thus $h=h'$ is in $K$.\\
{\bf{Case 2:}} $h$ is in $O$ and in $S$. In this case, by definition of $\Gamma$, the half-edge $h$ was the current half-edge when it was added to $S$. Thus, by Lemma~\ref{lemma:defGamma}, the half-edge $h$ is in $K$.\\
{\bf{Case 3:}} $h$ is in $I$ and in $\bS$. In this case,  one has $\beta(h)=\si(\al(h))=\theta(\al(h))$, hence $h'=\al(h)$. Since $h'$ is in $K$, Lemma~\ref{lemma:defGamma} ensures that it was the current half-edge at a certain step of the procedure $\Gamma$. Hence, since $h'=\al(h)$ is in $O$ but not in $S$, it means that $h$ and $h'$ were added to the set $R$ at a step of the procedure $\Gamma$, such that $h$ was the current half-edge. Thus,  by Lemma~\ref{lemma:defGamma}, the half-edge $h$ is in $K$.\\
{\bf{Case 4:}} $h$ is in $I$ and in $S$. In this case, one has $\beta(h)=\si\al(h)=\theta(h)$, hence $h'=h$.  Thus $h=h'$ is in $K$.
\end{proof}

We now complete the proof of Theorem~\ref{thm:Delta}.

\begin{proposition}
The mappings $\Delta$ and $\Gamma$ are inverse bijections between covered maps and left-connected maps.
\end{proposition}

\begin{proof}
\newcommand{\It}{\tilde{I}}
\newcommand{\Ot}{\tilde{O}}
We first prove that the mapping $\Delta\circ\Gamma$ is the identity on left-connected maps. Observe that this composition is well defined  by Proposition~\ref{prop:returnscovered}. Let $(M,(I,O))$ be a left-connected map and let $(M,S)$ be its image by $\Gamma$. We want to prove that $(M,(\It,\Ot))\equiv \Delta(M,S)$ is equal to $(M,(I,O))$. For that, it suffices to show that any half-edge $h$ such that $h\lS \al(h)$ is in $O$ if and only if it is in $\Ot$. Let $h$ be such a half-edge.  By definition of $\Delta$, it follows that $h$ is in $\Ot$ if and only if $h$ is in $S$. Now, by Lemma~\ref{lemma:defGamma}, the sequence $h_1=r,h_2,\ldots,h_{2n}$ of current half-edge visited during the procedure $\Gamma$ satisfies $h_1\lS h_2\lS \cdots\lS h_{2n}$, hence $h$ is visited before $\al(h)$ during the procedure $\Gamma$. Hence, by definition of $\Gamma$, the half-edge $h$ is in $O$ if and only if $h$ is in $S$.
  
We now prove that  the mapping $\Gamma\circ\Delta$ is the identity on covered maps. Observe that this composition is well defined and returns a covered map by Propositions~\ref{prop:returnsleftconnected} and~\ref{prop:returnscovered}.
 Let $(M,S)$ be a covered map with root $r$ and let $(M,(I,O))$ be its image by $\Delta$. Let also $(M,S')=\Gamma(M,(I,O))$ and let $\theta$ and $\theta'$ be respectively the motion function of $(M,S)$ and $(M,S')$. In order to prove that $S=S'$ it suffices to prove that $\theta=\theta'$ (indeed, the set $S$ and $S'$ are completely determined by $\theta$ and $\theta'$). Suppose now that $\theta\neq \theta'$ and consider the smallest integer $k\geq 0$ such that $\theta^{k+1}(r)\neq {\theta'}^{k+1}(r)$ (such an integer exists since~$\theta$ and~$\theta'$ are cyclic). Observe that for all $0\leq j<k$, the half-edge $\theta^j(r)={\theta'}^j(r)$ is in $S$ if and only if it is in $S'$ (since $\theta^{j+1}(r)={\theta'}^{j+1}(r)$). On the other hand, the half-edge $h=\theta^k(r)={\theta'}^k(r)$ is in the symmetric difference of $S$ and $S'$ (since $\theta^{k+1}(r)\neq {\theta'}^{k+1}(r)$). This implies that $h\lS \al(h)$ and  $h\lSp \al(h)$.
Since $h\lS \al(h)$, the definition of $\Delta$ shows that  the half-edge $h$ is in $S$ if and only if $h$ is in $O$. 
Since $h\lSp \al(h)$, Lemma~\ref{lemma:defGamma} proves that $h$ is the current half-edge before $\al(h)$ in the procedure $\Gamma$ on $(M,(I,O))$. Hence, by definition of $\Gamma$, the half-edge $h$ is in $S'$ if and only if $h$ is in $O$. This proves that $h=\theta^k(r)$ is not in the symmetric difference of $S$ and $S'$, a contradiction.
\end{proof}

Before leaving the world of left-connected maps, we prove the following result.
\begin{lemma}\label{lem:some-ingoing}
If $(M,(I,O))$ is a left-connected map with root $r$, then every non-root vertex of $M$ is incident to a half-edge in $I$ and every non-root face  is incident to a half-edge in $O$. 
\end{lemma}

\begin{proof}
If a cycle of the vertex-permutation $\si$ contains no edge in $I$, then $\be(h)=\si(h)$ for every half-edge $h$ in this cycle. By Lemma~\ref{lem:equiv-left-connected}, this implies that the root $r$ belongs to this cycle. Similarly, if a cycle of the face-permutation $\phi$ contains no edge in $O$, then $\be(h)=\phi(h)$ for every half-edge $h$ in this cycle. By Lemma~\ref{lem:equiv-left-connected}, this implies that the root $r$ belongs to this cycle. 
\end{proof}
\smallskip


\subsection{Proofs concerning the unfolding step.} 
In this Subsection we prove  Theorem~\ref{thm:bij}. 
We fix a left-connected map $(M,(I,O))$, where the map $M=(H,\si,\alpha)$ has $n$ edges, genus $g$, root $r$ and face-permutation $\phi$. We denote $A=(H,\tau,\al)=\Lambda_1(M,(I,O))$ and $B=(H',\pi,\al')=\Lambda_2(M,(I,O))$ and adopt the notation of Section~\ref{sec:bijection} for the sets $H'$, $I'$, $O'$ and the permutations $\si'$ $\phi'$, $\tau'$, $\pi_\circ$ and $\pi_\bu$. 

\begin{lemma}\label{lem:transitively}
The permutations $\tau,\al$ act transitively on $H$, thus $A$ is a map.
\end{lemma}

As mentioned above, the intuition behind the connectivity of $A$ is that left-paths are preserved by the unfolding. This can be formalized as follows.

\begin{proof} Let $\beta$ be the backward function of the oriented map $(M,(I,O))$ defined by $\beta(h)=\si(h)$ if $h$ is in $O$ and $\beta(h)=\si\al(h)$ otherwise. Let $\tbeta$ be the backward function of the oriented map $(A,(I,O))$ defined by $\tbeta(h)=\tau(h)$ if $h$ is in $O$ and $\tbeta(h)=\tau\al(h)$ otherwise.

We first prove that $\beta(h)=\tbeta(h)$ for any half-edge $h\in H$ such that $\beta(h)\neq r$.
If $h$ is in $O$ (and $\beta(h)\neq r$), then 
$$\beta(h)\equiv\si(h)=\si'(h)=\tau'(h)=\tau(h)\equiv\tbeta(h),$$ 
since the permutations $\si'$ and $\tau'$  coincide on $O'$. If $h$ is in $I$ (and $\beta(h)\neq r$), then  
$$\beta(h)\equiv\si(\al(h))=\si'(\al(h))=\tau'(\al(h))=\tau(\al(h))\equiv\tbeta(h),$$  
since again the permutations $\si'$ and $\tau'$ coincide on $O'$.

Since the oriented map $(M,(I,O))$ is left-connected, Lemma~\ref{lem:equiv-left-connected} ensures that for any half-edge $h$, there exists an integer $q>0$ such that $\beta^q(h)=r$. Taking the least such integer $q$, and using the preceding point shows that for any half-edge $h$, there exists an integer $q>0$ such that $\tbeta^{q-1}(h)=\beta^{q-1}(h)=l$, where $l$ is a half-edge such that $\beta(l)=r$. Moreover, the relation $\beta(l)=r$ shows that $l$ is either equal to $u=\si^{-1}(h)$ or to $\al(u)$. Therefore, any half-edge $h$ can be sent to one of the half-edges $u,\al(u)$ by applying the function $\tbeta$, hence by acting with the permutations $\tau$ and $\al$. This proves the lemma.
\end{proof}

We call \emph{root-to-leaves} orientation of a plane tree the unique orientation such that each non-root vertex is incident to exactly one ingoing half-edge.

\begin{proposition}\label{prop:tree-is-oriented}
The map $A$ is a tree. Moreover, $(I,O)$ is the root-to-leaves orientation of $A$.
\end{proposition}

We start with an easy lemma.
\begin{lemma} \label{lem:not-alone}
The half-edge $u=\si^{-1}(r)$ is in $O$ and $\tau'(u)=i$. In particular, the half-edge $i$ is not alone in its cycle of $\tau'$.
\end{lemma}

\begin{proof} 
Consider the  backward function $\beta$ of the oriented map $(M,(I,O))$. Since, $(M,(I,O))$ is left-connected,  Lemma~\ref{lem:equiv-left-connected} ensures that there exists a half-edge $h$ such that $\beta(h)=r$. Since $\beta(h)=\beta(\al(h))$ we can take $h$ in $O$ and get $h=\si^{-1}(r)=u$. This proves that $u=\si^{-1}(r)$ is in $O$. It is then obvious from the definition of $\tau'$ that $\tau'(u)=i$.
\end{proof}

\begin{proof}[Proof of Proposition~\ref{prop:tree-is-oriented}] 
Recall that $n=|H|/2$ denotes the number of edges of $M$, hence of $A$. We first prove that the map $A$ has (at least) $n+1$ vertices. By construction, the permutation $\tau'=\si'\pi_\circ^{-1}=\si'{\si'}_{|I}^{-1}$ has at most one element of $I'$ in each of its cycle. Hence it has at least $n+1$ cycles beside the cycle made of $o$ alone.  Moreover, by Lemma~\ref{lem:not-alone}, the half-edge $i$ is not alone in its cycle of $\tau'$, thus the vertex-permutation $\tau=\tau'_{|H}$ has at least $n+1$ cycles.

The (connected) map $A$ has $n$ edges and (at least) $n+1$ vertices hence it is a tree. Moreover, each non-root vertex of $A$ is incident to exactly one half-edge in $I$. Thus, $(I,O)$ is the  root-to-leaves orientation of $A$.
\end{proof}

We prove a last easy lemma about the tree $A=(H,\tau,\al)$.

\begin{lemma} \label{lem:not-alone2}
The permutation $\tau'=\si'\pi^{-1}_\circ$ is obtained from the vertex-permutation $\tau$ by inserting the half-edge $i$ before the root $t$ of $A$ in the cycle of $\tau$ containing $t$ and creating a cycle made of $o$ alone. The permutation $\varphi'=(i,o)\tau'\al'$ is obtained from the face-permutation $\varphi=\tau\al$ by inserting the half-edge $o$ before $t$ in the cycle of $\varphi$ containing $t$ and creating a cycle made of $i$ alone.  
\end{lemma}

\begin{proof} By definition, $\tau'(i)=t$ and $\tau'_{|H}=\tau$. Moreover, it is easy to check that  $\tau'(o)=\si'\pi^{-1}_\circ(o)=o$. This proves the statement about $\tau'$. We now denote $u={\tau'}^{-1}(i)$. By definition, $\varphi(\al(u))=\tau(u)=t$ and one can check $\varphi'(\al(u))=o$, $\varphi'(o)=t$, $\varphi'(i)=i$, and $\varphi'(h)=\varphi(h)$ for all $h\notin \{ i,o,\al(u)\}$. This proves the statement about $\varphi'$.
\end{proof}

We will now prove that the mobile $B=(H',\pi,\al')=\Lambda_2(M,(I,O))$ is a unicellular bipartite map. We introduce some notations for the half-edges in $H$. From Proposition~\ref{prop:tree-is-oriented},  the tree  $(A,(I,O))$ is oriented from root to leaves. In particular, the root $t$ of $A$ is outgoing. We denote $o_1=t,o_2,\ldots,o_n$ the outgoing half-edges as appearing during a counterclockwise tour around the tree $A$, that is to say, ${\varphi}_{|O}=(o_1,o_2,\ldots,o_n)$ where  $\varphi=\tau\al$ is the face-permutation of~$A$. This labelling is indicated in Figure~\ref{fig:coherent-labelling}(a). Observe that Lemma~\ref{lem:not-alone2} implies ${\varphi'}_{|O'}=(o,o_1,o_2,\ldots,o_n)$. For $j=1,\ldots,n$, we denote $i_j=\al(o_j)$, so that  $\al'{\varphi'}_{|O}\al'=(i,i_1,i_2,\ldots,i_n)$.\\

\begin{figure}[ht!]\begin{center} \input{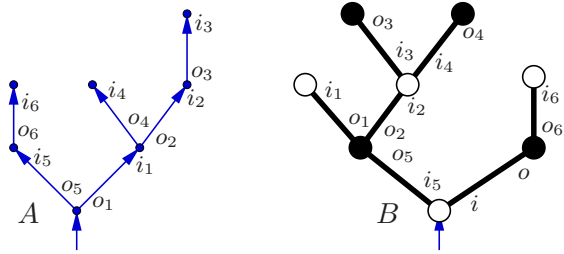}\caption{A pair $(A,B)$ coherently labelled on the alphabet $\{i,i_1,\ldots,i_6,o,o_1,\ldots, o_6\}$: the face-permutation $\varphi$ of the tree $A$ satisfies $\varphi'_{|O'}=(o,o_1,\ldots, o_6)$, while the face-permutation $\psi$ of the mobile $B$ satisfies $\psi_{|I'}=(i,i_1,\ldots, i_6)$.}\label{fig:coherent-labelling} \end{center}\end{figure}


\begin{proposition}\label{prop:mobile}
The mobile $B=\Lambda_2(M,(I,O))$ is a bipartite unicellular map of genus~$g(M)$. Moreover, if we color the vertices of $B$ in two colors, say white and black, with the root-vertex being white, then it has $v(M)$ white vertices and $f(M)$ black vertices. Lastly, the half-edges incident to white vertices are $i,i_1,i_2,\ldots,i_n$ and appear in this order during a clockwise tour around $B$, that is to say, ${\psi}^{-1}_{|I'}=(i,i_1,\ldots, i_n)=\al'{\varphi'}_{|O}\al'$, where $\psi=\pi\alpha'$ is the face-permutation of $B$ and $\varphi'=(i,o)\tau'\al'$. 
\end{proposition}

\begin{lemma}\label{lem:order-coincide}
The permutation $\psi=\pi\alpha'$ and $\varphi'=(i,o)\tau'\al'$ are related by ${\psi}^{-1}_{|I'}=(i,i_1,\ldots,i_n)=\al'{\varphi'}_{|O'}\al'$.
\end{lemma}

\begin{proof} 
We consider a half-edge~$h$ in~$I'$ and want to show $\al'{\varphi'}_{|O'}\al'(h)={\psi}^{-1}_{|I'}(h)$. 
First observe that for any permutation $\gam$ on $H'$, one has $\al'{\gam}_{|O'}\al'=(\al'\gam\al')_{I'}$ because the involution $\al'$ maps $I'$ to $O'$. In particular, $\al'{\varphi'}_{|O'}\al'=(\al'\varphi'\al')_{|I'}\equiv(\al'(i,o)\si'{\si'}^{-1}_{|I'})_{|I'}$.  Thus,
$$\al'{\varphi'}_{|O'}\al'(h)=(\al'(i,o)\si'{\si'}^{-1}_{|I'})_{|I'}(h)=(\al'(i,o)\si')_{|I'}{\si'}^{-1}_{|I'}(h),$$
since ${\si'}^{-1}_{|I'}$ acts as the identity on $O'$.

We now determine ${\psi}^{-1}_{|I'}(h)$. Observe that permutation ${\psi}^{-1}\equiv \al'\pi^{-1}$  maps  $I'$ to $O'$ (since $I'$ and $O'$ are stable by $\pi\equiv\pi_\circ\pi_\bu^{-1}$).  Thus, 
$${\psi}^{-1}_{|I'}(h)={\psi}^{-1}{\psi}^{-1}(h)\equiv\al'\pi^{-1}\al'\pi^{-1}(h)=\al'\pi_\bu\al'\pi_{\circ}^{-1}(h).$$
By definition, $\pi_\bu=\phi'_{|O'}$, $~\phi'=(i,o)\si'\al'$ and  $\pi_\circ=\si'_{|I'}$, hence 
$$\al'\pi_\bu\al'\pi_{\circ}^{-1}(h)\equiv \al'\phi'_{|O'}\al'\si_{|I'}^{-1}(h)=(\al'\phi'\al')_{|I'}\si_{|I'}^{-1}(h)\equiv (\al'(i,o)\si')_{|I'}\si_{|I'}^{-1}(h).$$
This gives 
$${\psi}^{-1}_{|I'}(h)=(\al'(i,o)\si')_{|I'}\si_{|I'}^{-1}(h)=\al'{\varphi'}_{|O'}\al'(h)$$
which concludes the proof.
\end{proof}

\begin{proof}[Proof of Proposition~\ref{prop:mobile}]
By Lemma~\ref{lem:order-coincide}, every half-edge in $I'$ belongs to the same cycle $C$ of the permutation $\psi=\pi\al'$. Now, if $h$ is a half-edge in $O'$, $\psi(h)=\pi\al(h)$ is in $I'$ (since $I'$ is stable by $\pi$) hence it belongs to the cycle $C$. Thus, the permutation $\psi$ is cyclic. Hence, the permutations $\pi$ and $\al$ act transitively on $H'$, that is, $B$ is a map. Moreover, its face-permutation $\psi$ is cyclic, that is, $B$ is unicellular. 

Moreover, since the sets $I'$ and $O'$ are stable by the vertex-permutation $\pi$ and exchanged by the edge-permutation $\al'$, the map $B$ is bipartite. Let us therefore consider the bipartite coloring where the vertices incident to half-edges in $I'$ are white while the vertices incident to the half-edges in $O'$ are black. By Lemma~\ref{lem:some-ingoing}, each of the cycles of $\si'$ except the cycle made of $o$ alone contains at least one half-edge in $I'$. Therefore, the number of cycles of the permutation $\pi_\circ=\si'_{|I'}$ on $I'$ is the number $v(M)$ of cycles of the vertex-permutation $\si$. Similarly, the number of cycles of the permutation $\pi_\bu=\phi'_{|O'}$ on $O'$ is the number $f(M)$ of cycles of the face-permutation $\phi$. Thus, the map $B$ has $v(M)$ white vertices and $f(M)$ black vertices. Now, Euler relation gives 
$$2g(B)=2+e(B)-f(B)-v(B)=2+(e(M)+1)-1-(v(M)+f(M))=2g(M).$$
Thus, the genus of $B$ is $g(M)$. This concludes the proof of Proposition~\ref{prop:mobile}.
\end{proof}

\noindent \textbf{Topological description of the folding step (Figure~\ref{fig:folding}).} 
We now define a mapping $\Omega$, the \emph{folding step}, which we will prove to be the inverse of the unfolding step~$\Lambda$. Before defining $\Omega$ in terms of permutations, let us explain the topological interpretation of Proposition~\ref{prop:mobile}. We denote by $v_0,v_1,\ldots,v_n$ the vertices of $A$ \emph{in counterclockwise order around $A$} (starting from the root-corner) and by $x_0,x_1,\ldots,x_n$ the \emph{first corners} of these vertices; see Figure~\ref{fig:folding}(a). 
Equivalently, $v_0$ is the root-vertex and $x_0$ is the root-corner, while for $j=1\ldots n$, $v_j$ is the vertex incident to the ingoing half-edge $i_j$ and $x_j$ is the corner following $i_j$ in counterclockwise order around $v_j$; see Figure~\ref{fig:coherent-labelling}. We also denote by $y_0,\ldots,y_n$ the white corners of $B$  \emph{in clockwise order around $B$} (starting from the root-corner $y_0$); see Figure~\ref{fig:folding}(a). By Proposition~\ref{prop:mobile}, for $j=0\ldots n$, $y_j$ is the corner of $B$ following the half-edge $i_j$ in clockwise order around the incident vertex. Therefore this proposition indicates how the folding step $\Omega=\Lambda^{-1}$ should be defined topologically: 
for $j=0,\ldots,n$ the first-corner $x_j$ of the vertex $v_j$ (the $j$th vertex of $A$ in counterclockwise direction) is glued to the corner $y_j$ (the $j$th white corner of $B$ in clockwise direction). This gives a map containing edges of both $A$ and $B$ that we call \emph{partially folded map} which is represented in Figure~\ref{fig:folding}(b). The oriented map   $(M,(I,O))=\Omega(A,B)$ is obtained from the partially folded map by keeping the half-edges of $A$ with their cyclic ordering around the white vertices of $B$ (while the edges and black vertices of $B$ are deleted).\\ 

\begin{figure}[ht!]\begin{center} \input{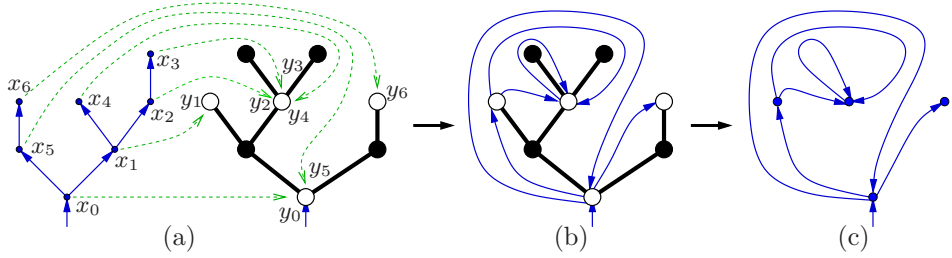}\caption{Topological representation of the folding step for the pair $(A,B)$. Figure~(a) indicates the correspondence between the first corners $x_0,\ldots,x_n$ of the tree $A$ and the white corners  $y_0,\ldots,y_n$ of the mobile $B$. Figure~(b) represents the partially folded map. Figure~(c) represents the oriented map   $(M,(I,O))=\Omega(A,B)$.}\label{fig:folding} \end{center}\end{figure}


We now defines the mapping $\Omega$ in terms of permutations. Let $\tA=(H,\ttau,\al)$ be a rooted plane tree with $n=|H|/2$ edges, and $\tB=(H',\tpi,\al')$ be a rooted bipartite unicellular map with $n+1=|H'|/2$ edges. We consider the usual black-and-white coloring of $\tB$ (with the root-vertex being white). The pair $(\tA,\tB)$ is said \emph{coherently labelled} if the following conditions are satisfied:  
\begin{itemize} 
\item[(i)] $H'=H\cup\{i,o\}$, where $i$ is the root of $M$ and $o=\al'(i)$.
\item[(ii)] $\al=\al'_{|H}$.
\item[(iii)]  The root-to-leaves orientation $(I,O)$ of $\tA$ is such that the half-edges in $I'=I\cup\{i\}$ are incident to white vertices of $\tB$, while half-edges in  $O'=O\cup\{o\}$ are incident to black vertices of $B$.
\item[(iv)] If the half-edges  $o_1,o_2,\ldots,o_n$ in $O$ appear in this order in counterclockwise direction around $\tA$ with $o_1$ being the root of $\tA$, then the half-edges  $i,i_1,\ldots,i_n$ defined by  $i_j=\al(o_j)$ for $j=1\ldots n$ appear in this order in clockwise direction around $\tB$. Equivalently, ${\psi}^{-1}_{|I'}=(i,i_1,\ldots, i_n)=\al'{\varphi'}_{|O}\al'$, where  $\psi=\pi\alpha'$ is the face-permutation of $B$ and $\varphi'=(i,o)\tau'\al'$. 
\end{itemize}

For example, the pair $(A,B)$ represented in Figure~\ref{fig:coherent-labelling} is coherently labelled. Observe that Proposition~\ref{prop:mobile} precisely states that the image by $\Lambda$ of a left-connected map $(M,(I,O))$ is coherently labelled. We now prove (the somewhat obvious fact) that any pair $(\tA,\tB)$ can be relabelled coherently.

\begin{lemma}\label{lem:formecanonique}
Let $\tA=(H,\ttau,\al)$ be a rooted plane tree with $n=|H|/2$ edges, and $\tB=(H',\tpi,\al')$ be a rooted bipartite unicellular map with $n+1=|H'|/2$ edges. Then, there is a unique way of relabelling $\tB$ in such a way that the pair $(\tA,\tB)$ is coherently labelled.
\end{lemma}

\begin{proof}
 We denote by $(I,O)$ the root-to-leaves labelling of $A$. We denote by $I'$ and $O'$ respectively the set of half-edges incident to white and black vertices of $\tB$ and observe that these sets are exchanged by $\al'$ (since $\tB$ is bipartite). We denote $\tvarphi=\ttau\al$ and $\tpsi=\tpi\al'$ are the face-permutation of $\tA$ and $\tB$ respectively. Lastly, we denote  $(o_1,o_2,\ldots,o_n)$ the cycle ${\tvarphi}_{|O}$ with $o_1$ being the root of $\tA$ and  $(i',i_1',\ldots,i_n')$ the cycle ${\tpsi}_{|I'}^{-1}$ with $i'$ being the root of $\tB$. Now we consider the relabelling of $\tB$ given by the bijection $\la$ from $H'$ to $H\cup\{o,i\}$ (where $i,o$ are half-edges not in $H$) given by $\la(i')=i$, $\la(\al(i'))=o$ and for all $j=1\ldots n$, $\la(i_j')=\al(o_j)$ and $\la(\al(i_j'))=o_j$. Clearly $\la$ is a bijection and it is the unique bijection making the pair $(\tA,\tB)$ coherently labelled.
\end{proof}

We denote by $P_n$ the set of pairs  $(\tA,\tB)$ made of a tree of size $n$ and a unicellular bipartite map of size $n+1$. 
We now consider such a pair $(\tA,\tB)$, where $\tA=(H,\ttau,\al)$ and $\tB=(H',\tpi,\al')$. By Lemma~\ref{lem:formecanonique}, we can assume that the pair $(\tA,\tB)$ is coherently labelled and we adopt the notations $i$, $o$, $I$, $O$, $I'$, $O'$ introduced in the conditions (i- iv) (in particular, $(I,O)$ is the root-to-leaves orientation of $A$).  We then define $\ttau'$ as the permutation on $H'$ obtained from $\tau$ by inserting $i$ before the root $t$ of $\tA$ in the cycle containing it and creating a cycle made of $o$ alone. We also define the permutations $\tpi_\circ$ and $\tsi'$ on $H'$ and the permutation $\tsi$ on $H$ by 
\begin{equation}\label{eq:tilde-pitau}
\tpi_\circ=\tpi_{|I'},~~~ \tsi'=\ttau'\tpi_\circ~~~ \textrm{ and }~~~ \tsi=\tsi'_{|H}~,
\end{equation}
(where a slight abuse of notation is done by considering $\tpi_\circ$ as a permutation on $H'$ acting as the identity on $O'$). With these notations, we define $\Omega(\tA,\tB)=(\tM,(I,O))$, where $\tM=(H,\tsi,\al)$.\\

We now complete the proof of Theorem~\ref{thm:bij} by proving the following proposition.

\begin{proposition}\label{prop:Lambda-Omega}
The mappings $\Lambda$ and $\Omega$ are inverse bijections between left-connected maps of size $n$ and pairs in $P_n$.
\end{proposition}

\begin{proof}
\ite We first prove that the mapping $\Omega\circ\Lambda$ is the identity on left-connected maps.\\
Let $(M,(I,O))$ be a left-connected map, where $M=(H,\sigma, \alpha)$. Let $(A,B)=\Lambda(M,(I,O))$, where $A=(H,\tau,\al)$ and $B=(H',\pi,\al')$ (recall that $(A,B)$ is coherently labelled by Proposition~\ref{prop:mobile}). Let also $(\tM,(\tI,\tO))=\Omega(A,B)$, where $\tM=(H,\tsi,\al)$. We want to prove that $(I,O)=(\tI,\tO)$ and $M=\tM$ (or equivalently, $\si=\tsi$).  
 
By definition of $\Omega$, $(\tI,\tO)$ is the root-to-leaves orientation of $A$. Moreover, by Proposition~\ref{prop:tree-is-oriented}, $(I,O)$ is also the root-to-leaves orientation of $A$. Hence, $(I,O)=(\tI,\tO)$.
   
By definition, $\si=\si'_{|H}$, where $\si'=\tau'\pi_\circ=\tau'\pi_{|I'}$ (see~\eqref{eq:pitau}). Similarly,  $\tsi=\tsi'_{|H}$ and $\tsi'=\ttau'\tpi_\circ=\ttau'\pi_{|I'}$ (see~\eqref{eq:tilde-pitau}). Moreover, Lemma~\ref{lem:not-alone2} ensures that $\tau'=\ttau'$ (since the permutations $\tau'$ and  $\tau''$ are obtained from $\tau$ by the same procedure). Thus, $\si'=\tsi'$ and $\si=\tsi$.\\

\ite We now prove that  the mapping $\Lambda\circ\Omega$ is the identity on $P_n$.\\
We must first prove that this mapping is well defined, that is, the image of any pair $(\tA,\tB)\in P_n$ by $\Omega$ is a left-connected map. Let us denote $\tA=(H,\ttau,\al)$ and $\tB=(H',\tpi,\al')$. 
By Lemma~\ref{lem:formecanonique}, we can assume that the pair $(\tA,\tB)$ is coherently labelled and we adopt the notations $i$, $o$, $I$, $O$, $I'$, $O'$ of conditions (i-iv) and the notations $\tpi_\circ$, $\tsi'$, $\tsi$ introduced in~\eqref{eq:tilde-pitau}. Lastly, we denote $\Omega(\tA,\tB)=(\tM,(I,O))$, where $\tM=(H,\tsi,\al)$. 

Let $\tbeta$ be the backward function of the tree $\tA$ defined on $H$ by $\tbe(h)=\ttau(h)$ if $h$ is in $O$ and $\tbe(h)=\ttau\al(h)$ otherwise. Let $\beta$ be defined on $H$ by $\beta(h)=\tsi(h)$ if $h$ is in $O$ and  $\beta(h)=\tsi\al(h)$ otherwise. Let $t$ be the root of $\tA$ and let $u=\tau^{-1}(t)$. It is easy to show (as is done in the proof of Lemma~\ref{lem:transitively}) that $\beta(h)=\tbeta(h)$ as soon as $\tbeta(h)\neq t$. The tree $(A,(I,O))$ is oriented from-root-to leaves, hence it is left-connected. Thus, by Lemma~\ref{lem:equiv-left-connected}, for any half-edge $h\in H$ there exists an integer $q>0$ such that $\be^q(h)=t$. By taking the least such integer $q$, one gets  $\tbe^{q-1}(h)=\be^{q-1}(h)\in \{u,\al(u)\}$. Since $u$ is in $O$, one gets $\tbe(\al(u))=\tbe(u)=\tsi(u)=r$. Hence, $\tbe^{q}(h)=r$. By Lemma~\ref{lem:equiv-left-connected}, this implies that  $(\tM,(I,O))$ is left-connected.\\

We now study the restrictions $\tpi_\circ\equiv\tpi_{|I'}$ and $\tpi_\bu\equiv\tpi^{-1}_{|O'}$. Recall that the map $\tB$ is bipartite with white vertices incident to half-edges in $I'$ and black vertices incident to half-edges in $O'$. Hence, $\tpi=\tpi_\circ\tpi_\bu^{-1}$.
 
We first prove that the permutation $\tpi_\circ$ is equal to $\tsi'_{|I'}$. We consider a half-edge $h$ in $I'$. By definition of restrictions, $\tsi'_{|I'}(h)={\tsip}^k(h)$ for a positive integer $k$ such that for all $0<j<k$, the half-edge ${\tsip}^j(h)$ is in $O'$. Moreover, the permutations $\ttau'$ and $\tsi'$ coincide on $O'$. Thus, $\tsi'_{|I'}(h)={\ttaup}^{k-1}(\tsi'(h))$. By~\eqref{eq:tilde-pitau}, $\tsi'=\ttau'\tpi_\circ$, hence $\tsi'_{|I'}(h)={\ttaup}^{k}(\tpi_\circ(h))$. Therefore,  $\tsi'_{|I'}(h)$ is a half-edge in $I'$ contained in the cycle of $\ttau'$ containing the half-edge $\pi_\circ(h)$ (which is in $I'$).  Since $(I,O)$ is the root-to-leaves orientation of $\tA$,  every cycle of $\ttau'$ contains exactly one half-edge in $I'$ (except for the cycle made of $o$ alone).  Thus,  $\tsi'_{|I'}(h)=\tpi_\circ(h)$.

We now prove that the permutation $\tpi_\bu$ is equal to $\tphi'_{|O'}$, where $\tphi'=(i,o)\tsi'\al'$. We consider the face-permutation $\tpsi=\tpi\al'$ of $\tB$. Since $\tpi=\tpi_\circ\tpi_\bu^{-1}$ one gets $\tpsi^{-1}=\al'\tpi_\bu\tpi_\circ^{-1}$, hence $\tpsi_{|I'}^{-1}=\al'\tpi_\bu\al'\tpi^{-1}_\circ$ and finally, $\tpi_\bu=\al'\tpsi_{|I'}^{-1}\tpi_\circ\al'$. We now use the property (iv) of the coherently labelled pair $(\tA,\tB)$. This property reads $\tpsi_{|I'}^{-1}=\al'\tvarphi'_{|O'}\al'$, where $\tvarphi'=(i,o)\ttau'\al'$. Thus,  $\tpi_\bu=\tvarphi'_{|O'}\al'\tpi_\circ\al'$. We now consider  a half-edge $h$ in $O'$. By definition of restrictions, $\tpi_\bu(h)=\tvarphip^k(\al'\tpi_\circ\al'(h))$, where $k$ is the least positive integer $k$ such that $\tvarphip^k(\al'\tpi_\circ\al'(h))$ is in $O'$.  Moreover, the permutations $\tvarphi'=(i,o)\ttau'\al'$  and $\tphi'=(i,o)\tsi'\al'$ coincide on $I'$ (since $\tsi'$ and $\ttau'$ coincide on $O'$). Thus,  $\tpi_\bu(h)=\tphip^{k-1}(\tvarphi'\al'\tpi_\circ\al'(h))$ , where $k$ is the least positive integer $k$ such that $\tphip^{k-1}(\tvarphi'\al'\tpi_\circ\al'(h))$ is in $O'$. Moreover, $\tvarphi'\al'\tpi_\circ\al'(h)=(i,o)\tsi'\al'(h)\equiv \tphi'(h)$ since $\tvarphi'= (i,o)\ttau'\al'$ and $\tpi_\circ=\ttau'^{-1}\tsi'$ by~\eqref{eq:tilde-pitau}. Thus, $\tpi_\bu(h)=\tphip^{k}(h)=\tphi'_{|O'}(h)$.

Given that $\tpi_\circ\equiv\tpi_{|I'}=\tsi'_{|I'}$ and $\tpi_\bu\equiv\tpi_{|O'}=\tphi'_{|O'}$, it is clear from the definition of the mapping $\Lambda$ that $\Lambda(\tM,(I,O))=(\tA,\tB)$. This concludes the proof of Proposition~\ref{prop:Lambda-Omega}.
\end{proof}

\smallskip


\section{Labeled and  blossoming mobiles, and link with known bijections}\label{sec:alternative-folding}
In this section we present two alternative ways of encoding the image of the bijection $\Psi$ (pairs made of a tree and a mobile) and we describe the bijection $\Psi$  in terms of these encodings. These alternative descriptions are particularly useful for studying the specializations of the bijection $\Psi$, and its relation with known bijections. In particular we will use them to prove that $\Psi$ can be specialized to a classical bijection by Bouttier \emph{et al}~\cite{BDFG:mobiles}. These descriptions are also used in~\cite{Bernardi-Fusy:dangulations,Bernardi-Fusy:Bijection-girth} in order to define a ``master bijection'' which can be specialized to a bijection $\Psi_d$, $d\geq 3$ for planar maps of girth $d$.

\subsection{Labeled mobiles and blossoming mobiles}
We first define the degree-code and height-code of a tree. Let $A$ be a rooted plane tree with $n$ edges. Let $v_0,v_1,\ldots,v_n$ be the vertices in counterclockwise order of appearance around the tree (with $v_0$ being the root). The \emph{height-code} of the tree $A$ is the sequence $c_0,\ldots,c_n$, where $c_j$ is the \emph{height} of the vertex $v_j$ (the number of edges on the path from $v_0$ to $v_j$). The \emph{degree-code} (or \emph{Lukasiewicz code}) is the sequence $d_0,\ldots,d_n$, where $d_j$ is the number of children of $v_j$. The height- and degree-code for the tree $A$ represented in Figure~\ref{fig:alternative} are respectively $0123212$ and $2210010$. It is well known that the height-code or degree-code both determine the tree $A$. Moreover, a sequence of non-negative integers $c_0,\ldots,c_n$ is a height-code if and only if
$$c_0=0~~~ \textrm{ and } ~~~\forall i<n,~ 0<c_{i+1}\leq c_i+1,$$
and a sequence of non-negative integers $d_0,\ldots,d_n$ is a degree-code if and only if
$$\sum_{i=0}^n(d_i-1)=-1 ~~~\textrm{ and }~~~ \forall k<n,~ \sum_{i=0}^k(d_i-1)\geq 0.$$

We can now give alternative descriptions of a pair $(A,B)$ by encoding the tree $A$ as some \emph{decorations} added to the mobile $B$ (the decorations corresponding either to the height- or degree-code of $A$). 
We consider the usual black and white coloring of the mobile $B$ (with the root-vertex being white). We say that the mobile $B$ is \emph{corner-labelled} if a non-negative number called \emph{label} is attributed to each of the $n+1$ white corners. The mobile $B$ is  \emph{well-corner-labelled} if the root-corner has label 0, all other corners have positive labels and the labels do not increase by more than 1 from a corner to the next one in clockwise direction around $B$. Equivalently, $B$ is well-corner-labelled if the sequence of corners encountered in clockwise order around the mobile starting from the root-corner is the height-code of a tree.   A well-corner-labelled mobile is shown in Figure~\ref{fig:alternative}(b). We now consider mobiles with \emph{buds}, that is, dangling half-edges. A \emph{blossoming mobile}  is a mobile $B$ together with some outgoing buds glued in each white corners. The  \emph{sequence of buds} encountered in clockwise order around the mobile is the sequence $d_0,d_1,\ldots,d_n$ where $d_i$ is the number of buds in the $i\tth$ corner of $B$ (in clockwise order starting from the root). The blossoming mobile is \emph{balanced} if its sequence of buds is the degree-code of a tree.   A balanced blossoming mobile is shown in Figure~\ref{fig:alternative}(c).
Since both the height- and degree-code (made of $n+1$ integers) determine a plane tree (with $n$ edges) the following result is obvious.

\begin{lemma}\label{lem:alternative-encoding} 
The three following sets are in bijection:
\begin{itemize}
\item pairs $(A,B)$ made of a plane tree $A$ and a mobile $B$ with respectively $n$ and $n+1$ edges, 
\item well-corner-labelled mobiles with $n+1$ edges, 
\item balanced blossoming mobiles  with $n+1$ edges.
\end{itemize}
\end{lemma}

\begin{figure}[ht!]\begin{center} \input{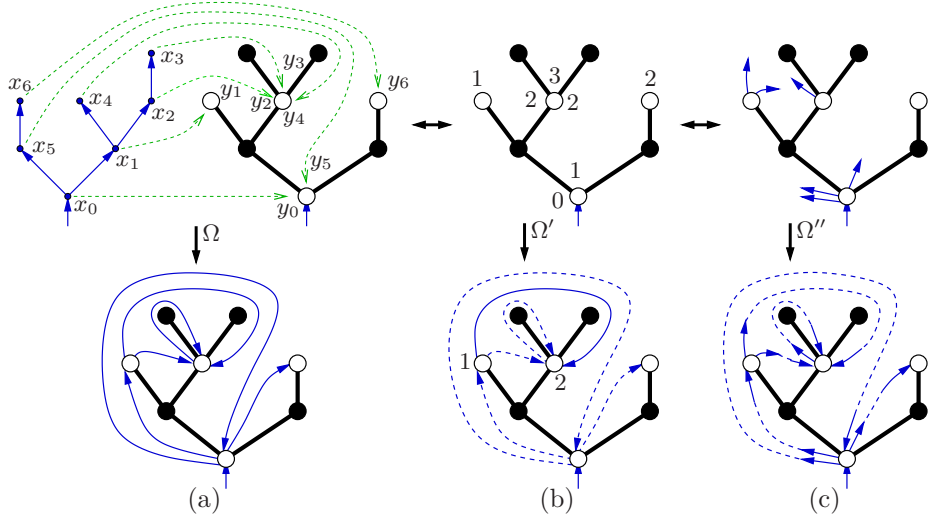}\caption{Three equivalent representations of a pair $(A,B)$ and descriptions of the folding step: (a) a pair consisting of a tree and a mobile; (b) a well-corner-labelled mobile; (c) a balanced blossoming mobile.}\label{fig:alternative} \end{center}\end{figure}

Lemma~\ref{lem:alternative-encoding} is illustrated in Figure~\ref{fig:alternative} (top part). By Theorem~\ref{thm:bij}, well-corner-labelled mobiles with $n+1$ edges, and balanced blossoming mobiles with $n+1$ edges are in bijections with covered maps. We now describe the folding step in terms of these objects.

Let $(A,B)$ be a pair made of a plane tree $A$ and a mobile $B$ with respectively $n$ and $n+1$ edges. Recall from Section~\ref{sec:proofs} the topological description of the \emph{folding step} $\Omega=\Lambda^{-1}$ for $(A,B)$:  if the vertices of the tree $A$ are denoted $v_0,v_1,\ldots,v_n$ in counterclockwise order around $A$ and the white corners of $B$ are denoted $y_0,y_1,\ldots,y_n$ in clockwise order around $B$, then the \emph{partially folded} map $N$ is obtained by gluing the first-corner $x_j$ of the vertex $v_j$ to the corner white $y_j$ of $B$ (see Figure~\ref{fig:folding}). The oriented map $(M,(I,O))=\Om(A,B)$ is then obtained from the partially folded map $N$ by deleting the edges and black vertices of $B$. 

Now let $(B,\ell)$ be the well-corner-labelled mobile,  where $\ell$ denotes the function associating a label to each white corner of $B$. 
For $j=1\ldots n$, we denote by $y_j'$ be the last corner of $B$ having label $\ell(y_j)-1$ appearing before $y_j$ in clockwise direction around $B$. Because  $(B,\ell)$ is well-labelled, the corner $y_j'$ always exists and appears between the root-corner and the corner $y_j$ in clockwise order around $B$.  The \emph{partially folded} map $N'$ associated to $(B,\ell)$ is defined as the map obtained from $B$ by adding $n$ edges sequentially: for $j=1\ldots n$, a directed edge $e_j$ of $N'$ is created from the corner $y_j'$ to the corner $y_j$ (there is a unique way to draw this edge without crossing). This procedure is represented in Figure~\ref{fig:alternative}(b). The (easy) proof of the following result is omitted:
\begin{prop}\label{prop:folding-labelled} If $(B,\ell)$ is the well-corner-labelled mobile corresponding with the pair $(A,B)$ (i.e. $\ell(y_0)\ell(y_1)\ldots\ell(y_n)$ is the height-code of $A$), then the partially folded maps $N$ and $N'$ coincide. 
\end{prop}

Let now  $\vec{B}$ be a balanced blossoming mobile. Let $\vec{B}'$ be the \emph{fully blossoming} mobile with ingoing and outgoing buds obtained from $\vec{B}$ by inserting an ingoing bud in each white corner of $\vec{B}$ following an edge of $B$ in clockwise order around the white vertex ($\vec{B}'$  is represented in solid lines in the bottom part of Figure~\ref{fig:alternative}(c)). 
Because the blossoming mobile $\vec{B}$ is balanced, the sequence of outgoing and ingoing buds in clockwise order around  $\vec{B}'$  (starting from the root-corner) is a parenthesis system (if outgoing and ingoing buds are seen respectively as opening and closing parentheses). Hence, there is a unique way of pairing each outgoing bud to an ingoing bud following it without creating any crossings. The \emph{partially folded} map $N'$ associated to the blossoming mobile $\bar{B}$ is the map obtained from $\vec{B}'$ by performing these pairings. The result is represented in Figure~\ref{fig:alternative}(c). 
The (easy) proof of the following proposition is omitted:
\begin{prop}\label{prop:folding-blossoming} If $\vec{B}$ is the blossoming mobile associated with the pair $(A,B)$ (i.e. the sequence of buds of $\vec{B}$ is the degree-code of $A$), then the partially folded maps $N$ and $N'$ coincide. 
\end{prop}

We mention that the description of the bijection $\Psi$ in terms of blossoming mobiles (which is used in \cite{Bernardi-Fusy:dangulations,Bernardi-Fusy:Bijection-girth}) is convenient for controlling the degrees of the covered maps: the degree of the vertices in a covered map corresponds to the degree of the white vertices in the associated blossoming mobile. 

\subsection{Link with the bijection of Bouttier, Di Francesco and Guitter~\cite{BDFG:mobiles}.} 
In~\cite{BDFG:mobiles} Bouttier, Di Francesco and Guitter defined a bijection between bipartite maps and \emph{well-labelled mobiles}\footnote{Strictly speaking, the bijection in~\cite{BDFG:mobiles} only describes the planar case. But is was explained in~\cite{Chapuy:constellations} how to extend it to higher genera.}. A \emph{well-labelled mobile} is a well-corner-labelled mobile such that the labels coincide around each white vertices, that is, any two corners incident to the same vertex have the same label. An example is given in Figure~\ref{fig:induced-BDG}(c). Observe that well-labelled mobiles are equivalently defined as mobiles with a label $\ell(v)$ associated to each vertex $v$ satisfying:
\begin{itemize}
\item the root-vertex has label 0 and degree 1, while other white vertices have positive labels,
\item the increase between the labels of two consecutive white vertices in clockwise order around a black vertex is at most 1. 
\end{itemize}

We now show that the bijection of Bouttier \emph{et al.} can be obtained as  a specialization of the \emph{unfolding mapping} $\Lambda'={\Omega'}^{-1}$.
We first recall some definitions. The \emph{distance} between two vertices of a map is the minimum number of edges on paths between them. We denote by $d(v)$ the distance of a vertex $v$ from the root-vertex. It is easy that for bipartite maps, any pair of adjacent vertices $u,v$ satisfies  $|d(u)-d(v)|=1$. Thus one can define the \emph{geodesic orientation} by requiring that the value of $d$ increases in the direction of every edge. The geodesic orientation is indicated in Figure~\ref{fig:induced-BDG}(b). We now state the main result of this subsection.
\begin{prop}
The geodesic orientation of a bipartite map is left-connected. Moreover, the unfolding mapping $\Lambda'$ induces a bijection between the set of bipartite maps (with $n$ edges and genus $g$) endowed with their geodesic orientation and the set of well-labelled mobiles (with $n+1$ edges and genus $g$). This  bijection coincides with the bijection described for bipartite maps in~\cite{BDFG:mobiles}. 
\end{prop}

\begin{proof}
Let $(M,(I,O))$ be a bipartite map endowed with its geodesic orientation. We first prove that the geodesic orientation is left-connected by using Lemma~\ref{lem:equiv-left-connected} concerning the backward function $\be$. Clearly, for any half-edge $h$ incident to a non-root vertex $v$, there exists an integer $p>0$ such that the half-edge $\beta^p(h)$ is incident to a vertex $u$ satisfying $d(u)=d(v)-1$ (because there are ingoing edges incident to $v$, and they all join $v$ to a vertex $u$ satisfying the property). Thus, there exists $q>0$ such that the half-edge $\beta^q(h)$ is incident to the root-vertex. Moreover, for any half-edge $h'$ incident to the root-vertex there exists an integer $r>0$ such that  the half-edge $\beta^r(h')$ is the root (because the root-vertex is only incident to outgoing half-edges). Therefore, by Lemma~\ref{lem:equiv-left-connected}, the geodesic orientation is left-connected. 
We now show that the corner-labelled mobile  $(B,\ell)=\Lambda'(M,(I,O))$ is  well-labelled. Let $v$ be a vertex of $M$ and let $v_1,\ldots,v_k$ be the vertices of the tree $A=\Lambda_1(M,(I,O))$ resulting from unfolding the vertex $v$. Clearly, any directed  path from the root-vertex to $v$ in $M$ has length $d(v)$. Hence, for all $i\in\{1,\ldots,k\}$ every directed  path from the root-vertex to $v_i$ in $A$ has length $d(v)$. Hence, the label $\ell$ of every corner of the white vertex $v$ of the mobile $B$ is equal to $d(v)$. Thus, the corner-labelled mobile $(B,\ell)$ is a well-labelled mobile.

Conversely, let $(B,\ell)$ be a well-labelled mobile, let $(M,(I,O))=\Omega'(B,\ell)$ be the corresponding left-connected map. We want to prove that $M$ is bipartite and $(I,O)$ is the geodesic orientation. By definition of the folding $\Omega'$ any edge of $M$ goes from a white vertex $u$ to a white vertex $v$ satisfying $\ell(v)=\ell(u)+1$. Hence, reasoning on the parity of labels shows that $M$ is bipartite. In order to prove that $(I,O)$ is the geodesic orientation, it suffices to prove that the label function $\ell$ is equal to the distance function $d$. Let $v$ be a non-root vertex. On one hand, one gets $d(v) \geq \ell(v)$ from the fact that labels cannot decrease by more than one when following an edge of $M$ (hence the root-vertex cannot be reached by following less than $\ell(v)$ edges). On the other hand, one gets $d(v)\leq \ell(v)$ from the fact that any non-root vertex of $M$ is adjacent to a vertex having a smaller label (by definition of the folding step $\Omega'$). Thus  $d=\ell$ and the orientation is geodesic.
\end{proof}

\begin{figure}[ht!]\begin{center} \input{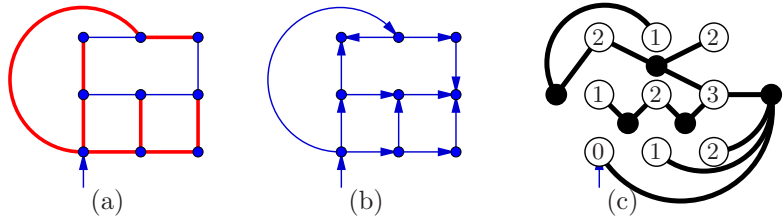}\caption{(a) The rightmost BFS tree. (b) The geodesic orientation. (c) The associated well-labelled mobile.}\label{fig:induced-BDG} \end{center}\end{figure}

In the remaining of this section, we complete the picture by characterizing the unicellular submap of a bipartite map which corresponds to the geodesic orientation (by the orientation step $\Delta$).
A spanning tree is said \emph{BFS} (for \emph{Breadth-First-Search}) if for any vertex $v$, the distance $d(v)$ is equal to the height in the spanning tree. 
\begin{definition}
The \emph{rightmost BFS tree} is the spanning tree $T$ obtained by the following procedure:\\
\titre{Initialization:} Set every vertex to be \emph{alive}. Set the tree $T$ as the tree containing the root-vertex of $M$ and no edge.\\ 
\titre{Core:} Consider the alive vertex $v$ which has been in the tree $T$ for the longest time and set it dead.
Inspect the half-edges incident to $v$ in counterclockwise order (starting from the root if $v$ is the root-vertex, and starting from the half-edge following the edge of $T$ leading $v$ to its parent otherwise) and whenever a half-edge leads to a vertex not in the tree $T$ add this vertex and the edge to $T$.\\
Repeat until all vertices are dead.\\
\titre{End:} Return the spanning tree $T$.
\end{definition}

The rightmost BFS tree is indicated in Figure~\ref{fig:induced-BDG}(a). We omit the (easy) proof of the following result.
\begin{prop}
Let $(M,S)$ be a bipartite covered map, and let $(M,(I,O))=\Delta(M,S)$ be the associated left-connected map. The orientation $(I,O)$ is geodesic if and only if $M_{|S}$ is the rightmost BFS tree of $M$.
\end{prop}

\smallskip


\section{Duality.}\label{sec:duality}
Recall from Section~\ref{sec:shuffles} that the dual of a covered map is a covered map. In this section, we explore the properties of the bijection $\Psi$ with respect to duality.  Throughout this section, we consider a covered map $(M,S)$, where the map $M=(H,\si,\al)$ has root $r$ and face-permutation $\phi=\si\al$. We denote $(M,(I,O))=\Delta(M,S)$ and $(A,B)=\Psi(M,S)$. 


\begin{lemma}[Duality at the orientation step]\label{lem:dualIO}
The oriented map associated to the dual covered map is the dual oriented map, that is to say, $\Delta(M^*,\bS)=(M^*,(O,I))$. 
\end{lemma}

Lemma~\ref{lem:dualIO} is illustrated in Figure~\ref{fig:dual+orientation}. 

\begin{proof}
Recall that the  the submaps $M_{|S}$ and $M^*_{|\bS}$ have the same motion functions, hence define the same appearance order on $H$. Thus, Lemma~\ref{lem:dualIO} immediately follows from the definition of the mapping~$\Delta$. 
\end{proof}

\begin{figure}[ht!]\begin{center} \input{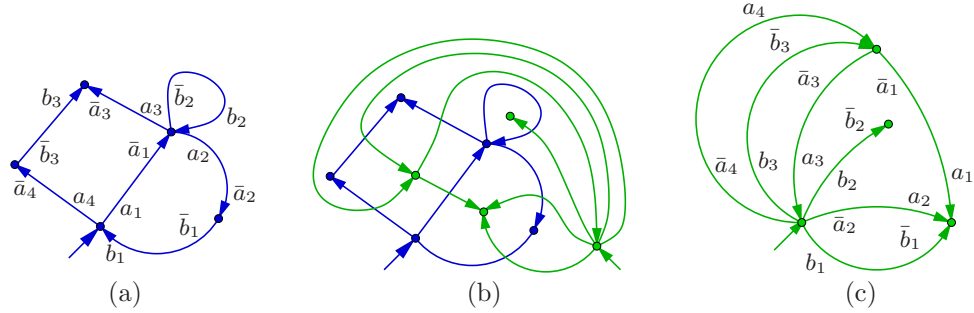}\caption{(a) The oriented map  $(M,(I,O))=\Delta(M,S)$ associated to the covered map represented in Figure~\ref{fig:dual+covered}(a). (b) Topological construction of the dual:  each oriented edge of $M$ is crossed by the the dual oriented edge of $M^*$ \emph{from left to right}. (c) The oriented map $(M^*,(O,I))$.}\label{fig:dual+orientation} \end{center}\end{figure}

We now explore the properties of the unfolding step with respect to duality. We denote $A=(H,\tau,\al)=\Psi_1(M,S)$ and $B=(H',\pi,\al')=\Psi_2(M,S)$, where $H'$ stands for $H\cup\{i,o\}$ and $i$ is the root of the mobile $B$. We also denote by $A^\st=(H,\tau^\st,\al)=\Psi_1(M^*,\bS)$ and $B^\st=(H',\pi^\st,\al')=\Psi_2(M^*,\bS)$, the plane tree and mobile associated to the dual covered map $(M^*,\bS)$. We shall prove the existence of two independent mappings $\Upsilon$ and $\Xi$ such that $A^\st=\Upsilon(A)$ and  $B^\st=\Xi(B)$.  In words, the duality acts \emph{component-wise} on the plane tree and the mobile.\\

\begin{figure}[ht!]\begin{center} \input{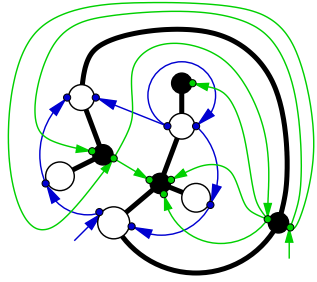}\caption{Simultaneous unfolding of the oriented map of Figure~\ref{fig:dual+orientation} and of its dual.}\label{fig:dual+unfolded} \end{center}\end{figure}

 
\begin{prop}[Duality and the mobile] \label{prop:dualmobile}
Let $(M,S)$ be a covered map and let $(M^*,\bS)$ be the dual covered map. If the mobile $B=\Psi_2(M,S)$ is denoted $(H',\pi,\al')$ and has root $i$, then the mobile $B^\st=\Psi_2(M^*,\bS)$ is the map $(H',\pi^{-1},\al')$ with root $o=\al(i)$. 
\end{prop}

Proposition~\ref{prop:dualmobile} is illustrated in  Figure~\ref{fig:dual-mobile}. It implies that the mobile $B^\st$ is entirely determined by the mobile $B$.

\begin{proof} 
The map $M$ has vertex-permutation $\si$ and face-permutation $\phi$, while the map $M^*$ has vertex-permutation $\si^\st=\phi$ and face-permutation $\phi^\st=\si$. We denote $\Delta(M,S)=(M,(I,O))$, so that $B=\Lambda(M,(I,O))$ and $B=\Lambda(M^*,(I^\st,O^\st))$, where $I^\st=O$ and $O^\st=I$ by Lemma~\ref{lem:dualIO}. We adopt the notations $i$, $o$, $I'$, $O'$, $\si'$, $\phi'$, $\pi_\circ$, $\pi_\bu$, $\pi$ of Section~\ref{sec:bijection} for defining $B$ and adopt the corresponding notations $i^\st$, $o^\st$, ${I'}^\st$, ${O'}^\st$, ${\si'}^\st$, ${\phi'}^\st$, ${\pi_\circ}^\st$, ${\pi_\bu}^\st$, ${\pi}^\st$ for defining $B^\st$. We choose $i^\st=o$ and $o^\st=i$, so that ${I'}^\st\equiv I^\st\cup\{i^\st\}=O'$, $~{O'}^\st\equiv O^\st\cup\{o^\st\}=I'$, $~{\si'}^\st=\phi'$ and $~{\phi'}^\st=\si'$. From this, it follows that  $\pi^\st_\circ\equiv{\si'^\st}_{|{I'}^\st}={\phi'}_{|O'}=\pi_\bu$ and $\pi^\st_\bu\equiv{\phi'}^\st_{|{O'}^\st}={\si'}_{|I'}=\pi_\circ$ and finally $\pi^\st\equiv\pi^\st_\bu {\pi^\st_\circ}^{-1}=\pi_\circ \pi_\bu^{-1}=\pi^{-1}$. Lastly, the root of $B^\st$ is $i^\st=o$.
\end{proof}

\begin{figure}[ht!]\begin{center} \input{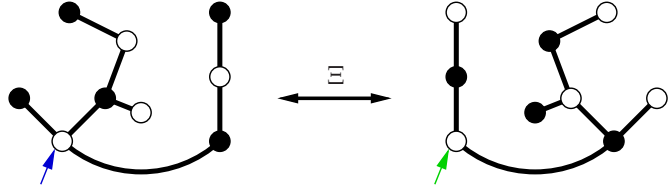}\caption{The mapping $\Xi$ between the mobile $B=\Psi_2(M,S)$ associated to the covered map $(M,S)$ of Figure~\ref{fig:dual+orientation} and the mobile $B^\st=\Psi_2(M^*,\bS)$ associated to the dual covered map.}\label{fig:dual-mobile} \end{center}\end{figure}

We now explicit the relation between the trees  $A$ and $A^\st$ in terms of their codes.

\begin{prop}[Duality and the tree]\label{prop:ups}
If the height-code of $A=\Psi_1(M,S)$ is $c_0,\ldots,c_n$, then the degree-code of $A^\st=\Psi_1(M^*,\bS)$ is  $d_0,\ldots,d_n$, where $d_{n-j}=c_{j}+1-c_{j+1}$ for $j=1,\ldots,n\!-\!1$ and $d_0=c_n$.
\end{prop}

Recall that a tree is completely determined by its height-code or by its degree-code. Hence, Proposition~\ref{prop:ups} shows that the tree $A^\st$ is entirely determined by the tree $A$.  Observe that the mapping $A\mapsto A^\st$ is an involution since duality of covered map is an involution. A topological version of this mapping is illustrated in Figure~\ref{fig:dual-tree2}(b), where the two trees $A$ and $A^\st$ are represented simultaneously in the way they interlace around the mobile's face. The rest of this section is devoted to the proof of Proposition~\ref{prop:ups}.\\

\begin{figure}[ht!]\begin{center} \input{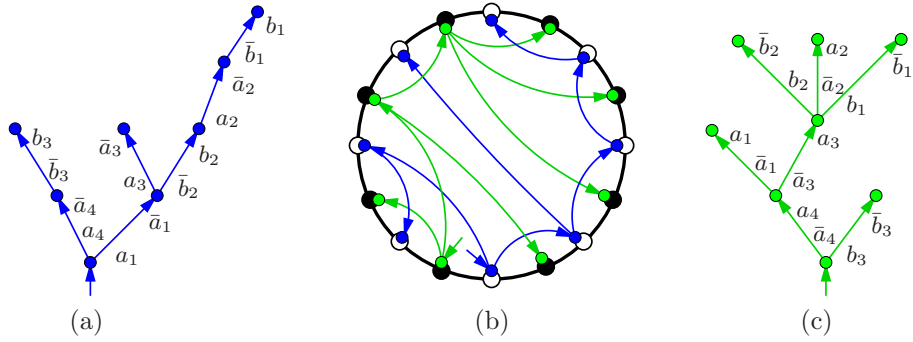}\caption{(a) The plane tree $A=\Psi_1(M,S)$ associated to the covered map of Figure~\ref{fig:dual+orientation}. (b) Topological construction of the tree $A^\st=\Psi_1(M^*,\bS)$: in the mobile face, the trees $A$ and $A^\st$ interlace in such a way that each edge of $A$ is crossed by exactly one edge of $A^\st$. (c) The plane tree $A^\st$.}\label{fig:dual-tree2} \end{center}\end{figure}

We denote by $t$ the root of the tree $A=(H,\tau,\al)$ and by $t^\st$ the root of the tree $A^\st=(H,\tau^\st,\al)$. 
We also adopt the notations  $\si'$, $\phi'$, $\pi_\circ$,  $\pi_\bu$, $\pi$, $\tau'$ of Section~\ref{sec:bijection} for the tree $A$ and adopt the corresponding notations ${\si'}^\st$, ${\phi'}^\st$, ${\pi_\circ}^\st$, ${\pi_\bu}^\st$, ${\pi}^\st$, $\tau'^\st$ for the tree $A^\st$. Lastly, we denote $\varphi=\tau\al$, $~\varphi'=(i,o)\tau'\al'$,  $~\varphi^\st=\tau^\st\al$ and $\varphi'^\st=(i,o)\tau'^\st\al'$. Recall that Lemma~\ref{lem:not-alone2} describes the (simple) link existing between the permutations $\tau$ and $\tau'$ and between $\varphi$  and $\varphi'$.

\begin{lemma}\label{lem:orderAA*}
The permutations $\varphi'$ and ${\varphi'}^\st$ are related by ${\varphi'}_{|O'}=(\al'{\varphi'}^\st\al')^{-1}_{|O'}$. 
\end{lemma}

For the example in Figure~\ref{fig:dual-tree2}, one gets ${\varphi'}_{|O'}=(o,a_1,\bb_2,a_2,\bb_1,a_3,a_4,\bb_3)$ and ${\varphi'}^\st_{|I'}=(b_3,\ba_4,\ba_3,b_1,\ba_2,b_2,\ba_1,i)$.

\begin{proof} 
By Proposition~\ref{prop:mobile}, the face-permutation $\psi=\pi\al$ of the mobile $B$ satisfies  ${\varphi'}_{|O'}=\al'\psi^{-1}_{|I'}\al'$. The same property applied to the mobile $B^\st$ gives ${\varphi'}^\st_{|I'}=\al'{\psi^\st}_{|O'}^{-1}\al'$, where $\psi^\st=\pi^\st\al'$ is the face-permutation of $B$. This gives 
$$(\al'{\varphi'}^\st\al')^{-1}_{|O'}=\al'({\varphi'}^\st)_{|I'}^{-1}\al'={\psi^\st}_{|O'}.$$
Moreover, by Proposition~\ref{prop:dualmobile}, $\pi^\st=\pi^{-1}$, so that $\psi^\st=\pi^{-1}\al'=\al'\psi^{-1}\al$. Hence,
$$(\al'{\varphi'}^\st\al')^{-1}_{|O'}= {\psi^\st}_{|O'}=(\al'\psi^{-1}\al)_{|O'}=\al'\psi^{-1}_{|I'}\al'={\varphi'}_{|O'}.$$
\end{proof}

\begin{lemma}\label{lem:dualtree}
The permutations $\varphi'$ and ${\tau'}^\st$ are related by ${\tau'}^\st=\varphi'{\varphi'}^{-1}_{|O'}$.
\end{lemma}

\begin{proof} 
By definition, ${\tau'}^\st=\si'^\st{\pi_\circ^\st}^{-1}=\phi'\phi'^{-1}_{|O'}$, where $\phi'=(i,o)\si'\al'$. We want to prove $\phi'{\phi'}^{-1}_{|O'}=\varphi'{\varphi'}^{-1}_{|O'}$, or equivalently, $\phi'_{|O'}\phi'^{-1}=\varphi'_{|O'}{\varphi'}^{-1}$ (by taking the inverse). Observe that the permutations $\phi'=(i,o)\si'\alpha'$ and $\varphi'=(i,o)\tau'\alpha'$ coincide on $I'$ (since $\si'$ and $\tau'$ coincide on $O'$). We now consider a half-edge $h$ in $H'$. Suppose first that $\phi'^{-1}(h)$ is in $I'$. 
In this case, ${\phi'}^{-1}(h)={\varphi'}^{-1}(h)$ (since $\phi'$ and $\varphi'$ coincide on $I'$), hence $\phi'_{|O'}\phi'^{-1}(h)={\phi'}^{-1}(h)={\varphi'}^{-1}(h)=\varphi'_{|O'}{\varphi'}^{-1}(h)$. Suppose now that  ${\phi'}^{-1}(h)$ is in~$O'$. Observe that ${\varphi'}^{-1}(h)$ is also in $O'$ (since $\phi'$ and $\varphi'$ coincide on $I'$). Moreover, by definition of restrictions, $\phi'_{|O'}\phi'^{-1}(h)={\phi'}^k(h)$, where $k\geq 0$ is such that ${\phi'}^{k}(h)\in O'$ and ${\phi'}^{j}(h)\in I'$ for all $0\leq j<k$. Since $\phi'$ and $\varphi'$ coincide on $I'$, we get ${\phi'}^j(h)={\varphi'}^j(h)$ for  $0\leq j\leq k$. Thus, $\phi'_{|O'}\phi'^{-1}(h)={\varphi'}^k(h)$ and where $k\geq 0$ is such that ${\varphi'}^{k}(h)\in O'$ and ${\varphi'}^{j}(h)\in I'$ for all $0\leq j<k$. Hence, by definition of restrictions, $\phi'_{|O'}\phi'^{-1}(h)=\varphi'_{|O'}{\varphi'}^{-1}(h)$.
\end{proof}

\smallskip

\begin{proof}[Proof of Proposition~\ref{prop:ups}]
We denote by $o_0=o,o_1,\ldots,o_n$ the half-edges in $O'$ in such a way that ${\varphi'}_{|O'}=(o,o_1,\ldots,o_n)$ and we denote $i_j=\al(o_j)$ for $j=0\ldots n$.  We denote by $v_0,v_1,\ldots,v_n$ the vertices of $A$ in counterclockwise order around  $A$. By definition (and because $(I,O)$ is the root-to-leaves orientation of $A$), this means that $v_j$ is incident to the ingoing half-edge $i_j$ for $j=1\ldots n$. Therefore, the height-code of $A$ is $c_0c_1\cdots c_n$, where $c_0=0$ and for $j=0\ldots n\!-\!1$, $c_{j+1}=c_j+1-\delta_j$ where $\delta_j$ is the number of half-edges in $I$ between the half-edge $o_j$ and $o_{j+1}$ in the face-permutation $\varphi$ (hence, also in the permutation $\varphi'$). Equivalently,  for $j=0\ldots n\!-\!1$, $\delta_j\equiv c_j+1-c_{j+1}$ is the number of half-edges in $I'$ in the cycle of the permutation $\varphi'{\varphi'}^{-1}_{|O'}$ containing $o_j$. We also denote $\delta_n=c_n$ and observe that this is the number of half-edges in $I'$ in the cycle of the permutation $\varphi'{\varphi'}^{-1}_{|O'}$ containing $o_n$.
 
We now consider degree-code $d_0d_1\cdots d_n$ of $A^\st$ and want to prove that $\delta_j=d_{n-j}$ for $j=0\ldots n$. Let $v_0^\st,v_1^\st,\ldots,v_n^\st$ be the vertices of $A^\st$ in counterclockwise order around  $A^\st$. By Lemma~\ref{lem:dualIO}, the root-to-leaves orientation of $A^\st$ is $(O,I)$ and by Lemma~\ref{lem:orderAA*},  ${\varphi'}^\st_{|I'}=(i,i_n,\ldots,i_2,i_1)$. Therefore, for $j=1\ldots n\!-\!1$ the vertex $v_{n-j}^\st$ of $A^\st$ is incident to the half-edge $o_j$. Thus, for $j=0\ldots n\!-\!1$,  the number of children $d_{n-j}$ of  $v_{n-j}^\st$ is the number of half-edges in $I'$ in the cycle of the vertex-permutation $\tau^\st$ containing $o_j$ (hence, also in the permutation $\tau'^\st$). For $j=n$ also, we observe that $d_{n-j}$ is the number of half-edges in $I'$ in the cycle of the permutation $\tau'^\st$ containing $o_j$. By Lemma~\ref{lem:dualtree}, $\tau'^\st=\varphi'{\varphi'}^{-1}_{|O'}$, hence  $\delta_j=d_{n-j}$ for $j=0\ldots n$. This concludes the proof of Proposition~\ref{prop:ups}.
\end{proof}

\titre{Acknowledgments.} We thank \'Eric Fusy, Jean-Fran\c cois Marckert, Gr\'egory Miermont, and Gilles Schaeffer for very stimulating discussions. 

\bibliographystyle{plain}
\bibliography{biblio-covered}

\end{document}